\documentclass[a4paper,11pt]{amsart}
\usepackage{amssymb, graphics, diagrams}
\input epsf
\newtheorem{lemma}{Lemma}

\newtheorem{prop}{Proposition}
\newtheorem{thm}{Theorem}
\newtheorem{cor}{Corollary}
\newtheorem{conj}{Conjecture}
\theoremstyle{definition}
\newtheorem{rem}{Remark}
\newtheorem{defn}{Definition}

\newcounter{numl}

\newcommand{\labelnuml}{\textup{(\roman{numl})}}
\newenvironment{numlist}{\begin{list}{\labelnuml}%
{\usecounter{numl}\setlength{\leftmargin}{0pt}%
\setlength{\itemindent}{2\parindent}%
\setlength{\itemsep}{\smallskipamount}\def
\makelabel ##1{\hss \llap {\upshape ##1}}}}{\end{list}}
\newenvironment{bulletlist}{\begin{list}{\labelitemi}%
{\setlength{\leftmargin}{0pt}\setlength{\itemindent}{\parindent}%
\setlength{\itemsep}{\smallskipamount}\def
\makelabel ##1{\hss \llap {\upshape ##1}}}}{\end{list}}

\newcommand{\R}{{\mathbb R}}
\newcommand{\C}{{\mathbb C}}

\newcommand{\T}{{\mathbb T}}% Torus
\newcommand{\cA}{{\mathcal A}}% Index set for base manifolds
\newcommand{\cB}{{\mathcal B}}% Blow-down index set
\newcommand{\cC}{{\mathcal C}}% Curve
\newcommand{\cE}{{\mathcal E}}% Line bundle
\newcommand{\cF}{{\mathcal F}}% Futaki
\newcommand{\cL}{{\mathcal L}}% Line bundle
\newcommand{\cM}{{\mathcal M}}% WBF function
\newcommand{\cO}{{\mathcal O}}% Line bundle on projective space
\newcommand{\cV}{{\mathcal V}}% Scheme
\newcommand{\cS}{{\mathcal S}}
\newcommand{\vE}{E}% Vector bundle
\newcommand{\Scal}{\mathit{Scal}}

\newcommand{\trace}{\mathop{\mathrm{tr}}\nolimits}

\newcommand{\ip}[1]{\langle #1 \rangle}
\newcommand{\del}{\partial}

\newcommand{\ang}{t}

\begin{document}
\title[Extremal K\"ahler metrics on projective bundles]
{Extremal K\"ahler metrics on projective bundles over a curve}

\author[V. Apostolov]{Vestislav Apostolov} \address{Vestislav Apostolov \\
D{\'e}partement de Math{\'e}matiques\\ UQAM\\ C.P. 8888 \\ Succ. Centre-ville
\\ Montr{\'e}al (Qu{\'e}bec) \\ H3C 3P8 \\ Canada}
\email{apostolov.vestislav@uqam.ca}

\author[D. Calderbank]{David M. J. Calderbank} \address{David M. J. Calderbank
\\ Department of Mathematical Sciences\\ University of Bath\\ Bath BA2 7AY\\
UK} \email{D.M.J.Calderbank@bath.ac.uk}

\author[P. Gauduchon]{Paul Gauduchon} \address{Paul Gauduchon \\ Centre de
Math\'ematiques\\ {\'E}cole Polytechnique \\ UMR 7640 du CNRS \\ 91128
Palaiseau \\ France} \email{pg@math.polytechnique.fr}

\author[C. T\o nnesen-Friedman]{Christina W.~T\o nnesen-Friedman}
\address{Christina W. T\o nnesen-Friedman\\ Department of Mathematics\\ Union
College\\ Schenectady\\ New York 12308\\ USA } \email{tonnesec@union.edu}

\begin{abstract} Let $M=P(E)$ be the complex manifold underlying the total
space of the projectivization of a holomorphic vector bundle $E \to \Sigma$
over a compact complex curve $\Sigma$ of genus $\ge 2$. Building on ideas
of Fujiki~\cite{fujiki}, we prove that $M$ admits a K\"ahler metric of
constant scalar curvature if and only if $E$ is polystable. We also address
the more general existence  problem of  extremal K\"ahler metrics on such
bundles and prove that the splitting of $E$ as a direct sum of stable
subbundles is necessary and sufficient condition for the existence of
extremal K\"ahler metrics in sufficiently small K\"ahler classes. The
methods used to prove the above results apply to a wider class of
manifolds, called {\it rigid toric bundles over a semisimple base}, which
are fibrations associated to a principal torus bundle over a product of
constant scalar curvature K\"ahler manifolds with fibres isomorphic to a
given toric K\"ahler variety. We discuss various ramifications of our
approach to this class of manifolds.
\end{abstract}
\maketitle

\section{Introduction}

Extremal K\"ahler metrics were first introduced and studied by E.~Calabi in
\cite{cal-one,cal-two}. Let $(M, J)$ denote a connected compact complex
manifold of complex dimension $m$.  A K\"ahler metric $g$ on $(M, J)$, with
K\"ahler form $\omega = g (J \cdot, \cdot)$, is {\it extremal} if it is a
critical point of the functional $g \mapsto \int _M s _g ^2 \, \frac{\omega
_g ^m}{m!}$, where $g$ runs over the set of all K\"ahler metrics on $(M,J)$
within a fixed K\"ahler class $\Omega = [\omega]$, and $s_g$ denotes the
scalar curvature of $g$. As shown in \cite{cal-one}, $g$ is extremal if and
only if the symplectic gradient $K := {\rm grad} _{\omega} s _g = J \, {\rm
grad} _g s _g$ of $s _g$ is a Killing vector field (i.e. $\mathcal{L} _K g
= 0$) or, equivalently, a (real) holomorphic vector field
(i.e. $\mathcal{L} _K J = 0$). Extremal K\"ahler metrics include K\"ahler
metrics of constant scalar curvature --- CSC K\"ahler metrics for short ---
in particular K\"ahler--Einstein metrics.  Clearly, if the identity
component ${\rm Aut} _0 (M, J)$ of the automorphism group of $(M, J)$ is
reduced to $\{1\}$, i.e. if $(M, J)$ has no non-trivial holomorphic vector
fields, any extremal K\"ahler metric is CSC, whereas a CSC K\"ahler metric
is K\"ahler--Einstein if and only if $\Omega$ is a multiple of the (real)
first Chern class $c _1 (M, J)$.  In this paper, except for
Theorem~\ref{main} below, we will be mainly concerned with extremal
K\"ahler metrics of {\it non-constant} scalar curvature.

The {\it Lichnerowicz--Matsushima theorem} provides an obstruction to the
existence of CSC K\"ahler metrics on $(M, J)$ in terms of the structure of
${\rm Aut} _0 (M, J)$, which must be {\it reductive} whenever $(M, J)$
admits a CSC K\"ahler metric; in particular, for any CSC K\"ahler metric
$g$, the identity component ${\rm Isom} _0 (M, g)$ of the group of
isometries of $(M, g)$ is a {\it maximal} compact subgroup of $(M,
J)$~\cite{matsushima,lichne}. The latter fact remains true for any extremal
K\"ahler metric (although ${\rm Aut} _0 (M, J)$ is then no longer reductive
in general) and is again an obstruction to the existence of extremal
K\"ahler metrics~\cite{cal-two,levine}.  Another well-known obstruction to
the existence of CSC K\"ahler metrics within a given class $\Omega$
involves the {\it Futaki character}~\cite{futaki, cal-two}, of which a
symplectic version, as developed in \cite{lejmi}, will be used in this
paper (cf. Lemma~\ref{symplectic-futaki}).  Furthermore, it is now known
that extremal K\"ahler metrics within a fixed K\"ahler class $\Omega$ are
{\it unique} up to the action of the {\it reduced automorphism
group}\footnote{$\widetilde{{\rm Aut}} _0 (M, J)$ is the unique {\it linear
algebraic subgroup} of ${\rm Aut} _0 (M, J)$ such that the quotient ${\rm
Aut} _0 (M, J)/\widetilde{{\rm Aut}} _0 (M, J)$ is a torus, namely the
Albanese torus of $(M, J)$~\cite{fujiki-0}; its Lie algebra is the space of
(real) holomorphic vector fields whose zero-set is non-empty
~\cite{fujiki-0,kobayashi,Le-Sim,gauduchon-book}.}  $\widetilde{{\rm Aut}}
_0 (M, J)$~\cite{BM,CT,Do-one,mab-two}.

It was suggested by S.-T. Yau \cite{yau} that a complete obstruction to the
existence of extremal K\"ahler metrics in the K\"ahler class $\Omega = c _1
(L)/2 \pi$ on a projective manifold $(M, J)$ polarized by an ample
holomorphic line bundle $L$ should be expressed in terms of {\it stability}
of the pair $(M, L)$.  The currently accepted notion of stability is the
$K$-({\it poly}){\it stability} introduced by G.~Tian~\cite{Tian2} for Fano
manifolds and by S.~K.~Donaldson~\cite{Do2} for general projective
manifolds polarized by $L$.  The {\it Yau--Tian--Donaldson conjecture} can
then be stated as follows. {\it A polarized projective manifold $(M, L)$
admits a CSC K\"ahler metric if and only if it is $K$-polystable.}  This
conjecture is still open, but the implication `CSC $\Rightarrow$
{K-polystable}' in the conjecture is now well-established, thanks to work
by S.~K.~Donaldson \cite{Do-one}, X. Chen--G. Tian~\cite{CT}, J.~Stoppa
\cite{stoppa}, and T.~Mabuchi \cite{mab-three,mab-three1}.  In order to
account for extremal K\"ahler metrics of non-constant scalar curvature,
G.~Sz\'ekelyhidi introduced \cite{Sz, gabor} the notion of {\it
relative} $K$-(poly)stability with respect to a maximal torus of the
automorphism group of the pair $(M, L)$ --- which is the same as the
reduced automorphism group $\widetilde{{\rm Aut}} _0 (M, J)$ --- and the
similar implication `extremal $\Rightarrow$ {relatively K-polystable}' was
recently established by G.~Sz\'ekelyhidi--J. Stoppa~\cite{gabor-stoppa}.

The Yau--Tian--Donaldson conjecture was inspired by and can be regarded as
a non-linear counterpart of the well-known equivalence for a holomorphic
vector bundle over a compact K\"ahler manifold $(M, J, \omega)$ to be {\it
polystable} with respect to $\omega$ on the one hand and to admit a {\it
hermitian--Einstein metric}~\cite{kobayashi, uy, Do0} on the other.
In the case when $(M, J)$ is a Riemann surface this is the celebrated
theorem of Narasimhan and Seshadri \cite{NS}, which, in the geometric
formulation given in \cite{AB,Do-NS,fujiki} can be stated as follows. {\it
Let $E$ be a holomorphic vector bundle over a compact Riemann surface
$\Sigma$. Then, $E$ is polystable if and only if it admits a hermitian
metric whose Chern connection is projectively-flat}.

\smallskip

This paper is mainly concerned with the existence of extremal K\"ahler
metrics on {\it ruled manifolds}, i.e. when $(M, J)$ is the total space of
projective fibre bundles $P (E)$ where $E$ is a holomorphic vector bundle
over a compact K\"ahler manifold $(S, J_S, \omega_S)$ of constant scalar
curvature.  Notice that this class of complex manifolds includes most
explicitly known examples so far of extremal K\"ahler manifolds of
non-constant scalar curvature, starting with the first examples given by
E. Calabi in \cite{cal-one}.  For complex manifolds of this type one
expects stability properties of $(M, J)$ to be reflected in the stability
of the vector bundle $E$. In fact, such a link was established by
J.~Ross--R.~Thomas, with sharper results when the base is a compact Riemann
surface of genus at least $1$,~\cite[Theorem 5.12 and Theorem 5.13]{RT}
(cf. also Remark \ref{rem2} below.)  This suggests that the existence of
extremal K\"ahler metrics on $P(E)$ could be directly linked to the
stability of the underlying bundle $E$. First evidence of such a direct
link goes back to the work of Burns--deBartolomeis~\cite{BdB}, and many
partial results in this direction are now known, see e.g. the works of
N.~Koiso--Y.~Sakane~\cite{koi-sak-one, koi-sak-two},
A.~Fujiki~\cite{fujiki}, D.~Guan~\cite{guan}, C.~R.~LeBrun~\cite{LeB},
C. T\o nnesen-Friedman~\cite{christina-one}, Y.-J.~Hong~\cite{hong, hong2},
A.~Hwang--M.~Singer~\cite{hwang-singer}, Y.~Rollin--M.~Singer~\cite{RS},
J.~Ross--R.~Thomas~\cite{RS}, G.~Sz\'ekelyhidi~\cite{Sz, gabor}, and our
previous work~\cite{ACGT}. However, to the best of our knowledge, an
understanding of the precise relation is still to come.

While we principally focus on projective bundles, for which sharper results
can be obtained (Theorems~\ref{main}, \ref{th:extremal}, \ref{th:existence}
below), the techniques and a number of results presented in this paper
actually address a much wider class of manifolds, called {\it rigid toric
bundles over a semisimple base}, which were introduced in our previous
paper \cite{hfkg2}. Section~\ref{calabi-type} of this paper is devoted to
recalling the main features of this class of manifolds and proving a
general existence theorem (Theorem~\ref{th:small-classes}).

\smallskip

The simplest situation considered in this paper is the case of a projective
bundle over a {\it curve}, i.e. a compact Riemann surface.  In this case,
the existence problem for CSC K\"ahler metrics can be resolved.
\begin{thm}\label{main} Let $(M,J)=P(E)$ be a holomorphic projective bundle
over a compact complex curve of genus $\ge 2$.  Then $(M,J)$ admits a CSC
K\"ahler metric in some {\rm (}and hence any{\rm )} K\"ahler class if and only
if the underlying holomorphic vector bundle $E$ is polystable.
\end{thm}
\begin{rem}\label{rem1} {\rm 
The `if' part follows from the theorem of Narasimhan and Seshadri: if $E$
is a polystable bundle of rank $m$ over a compact curve (of any genus),
then $E$ admits a hermitian--Einstein metric which in turn defines a flat
$PU(m)$-structure on $P(E)$ and, therefore, a family of locally-symmetric
CSC K\"ahler exhausting the K\"ahler cone of $P(E)$, see
e.g.~\cite{kobayashi}, \cite{fujiki}.  Note also that in the case when
$P(E)$ fibres over ${\mathbb C} P^1$, $E$ splits as a direct sum of line
bundles, and the conclusion of Theorem~\ref{main} still holds by the
Lichnerowicz--Matsushima theorem, see e.g.~\cite[Prop.~3]{ACGT}.}
\end{rem}

\begin{rem} \label{rem2} {\rm On all manifolds considered in
Theorem~\ref{main}, rational K\"ahler classes form a dense subset in the
K\"ahler cone.  By LeBrun--Simanca stability theorem~\cite[Thm.~A]{Le-Sim1}
and Lemma~\ref{futaki-inv} below it is then sufficient to consider the
existence problem only for an integral K\"ahler class (or polarization). In
this setting, it was shown by Ross--Thomas that any projective bundle $M =
P(E)$ over a compact complex curve of genus $\geq 1$ is K-poly(semi)stable
(with respect to some polarization) if and only if $E$ is
poly(semi)stable~\cite[Thm.~5.13]{RT}.  In view of this theorem, the ``only
if'' part of Theorem~\ref{main} can therefore be alternatively recovered
---for any genus $\geq 1$--- as a consequence of recent papers by
T.~Mabuchi \cite{mab-three,mab-three1}.  }\end{rem}

By the de Rham decomposition theorem, an equivalent differential geometric
formulation of Theorem~\ref{main} is that any CSC K\"ahler metric on
$(M,J)$ must be locally symmetric (see \cite[Lemma 8]{fujiki} and
\cite{LeB}). It is in this form that we are going to achieve our proof of
Theorem~\ref{main}, building on the work of A. Fujiki~\cite{fujiki}.  In
fact, \cite{fujiki} already proves Theorem~\ref{main} in the case when the
underlying bundle $E$ is {simple}, modulo the uniqueness of CSC K\"ahler
metrics, which is now known~\cite{CT,Do-one,mab-two}.

In view of this, the main technical difficulty in proving Theorem 1 is
related to the existence of automorphisms on $(M,J)=P(E)\mapsto
\Sigma$. The way we proceed is by fixing a maximal torus ${\T}$ (of
dimension $\ell$) in the identity component ${\rm Aut}_0(M,J)$ of the
automorphism group, and showing that it induces a decomposition of $E=
\bigoplus_{i=0}^{\ell} E_i$ as a direct sum of $\ell+1$ indecomposable
subbundles $E_i$, such that $\T$ acts by scalar multiplication on each
$E_i$ (see Lemma~\ref{decompose} below). By computing the Futaki invariant
of the $S^1$ generators of $\T$, we show that the slopes of $E_i$ must be
all equal, should a CSC K\"ahler metric exist on $P(E)$ (see
Lemma~\ref{futaki-inv} below).\footnote{For rational K\"ahler classes, this
conclusion can be alternatively reached by combining \cite[Thm.~5.3]{RT}
and \cite{Do4}.} Then, following the proof of \cite[Theorem~3]{fujiki}, we
consider small analytic deformations $E_i(t)$ of $E_i=E_i(0)$ with $E_i(t)$
being stable bundles for $t \neq 0$. This induces a $\T$-invariant
Kuranishi family $(M,J_t) \cong P(E(t)),$ where
$E(t)=\bigoplus_{i=0}^{\ell} E_i(t)$, with $(M,J)$ being the central fibre
$(M,J_0)$. We then generalize in Lemma~\ref{stability} the
stability-under-deformations results of \cite{Le-Sim,Le-Sim1, Fuj-Sch}, by
using the crucial fact that our family is invariant under a fixed {\it
maximal} torus. This allows us to show that any CSC (or more generally
extremal) K\"ahler metric $\omega_0$ on $(M,J_0)$ can be included into a
smooth family $\omega_t$ of extremal K\"ahler metrics on $(M,J_t)$. As
$E(t)$ is polystable for $t \neq 0$, the corresponding extremal K\"ahler
metric $\omega_t$ must be locally symmetric, by the uniqueness
results~\cite{CT, Do-one, mab-one}. This implies that $\omega_0$ is locally
symmetric too, and we conclude as in \cite[Lemma 8]{fujiki}.

\bigskip

We next consider the more general  problem of existence
of extremal K\"ahler metrics on the manifold $(M,J)= P(E)\to \Sigma$. Notice
that the deformation argument explained above is not specific to the CSC case,
but also yields that any extremal K\"ahler metric $\omega_0$ on $(M,J)=
P(E)\to \Sigma$ can be realized as a smooth limit (as $t \to 0$) of extremal
K\"ahler metrics $\omega_t$ on $(M,J_t) = P (E(t))$, where $E(t)
=\bigoplus_{i=0}^{\ell} E_i(t)$ with $E_i(t)$ being stable (and thus
projectively-flat and indecomposable) bundles over $\Sigma$ for $t\neq 0$, and
where $\ell$ is the dimension of a maximal torus $\T$ in the identity
component of the group of isometries of $\omega_0$. Unlike the CSC case (where
$E_i(t)$ must all have the same slope and therefore $E(t)$ is polystable), the
existence problem for extremal K\"ahler metrics on the manifolds $(M,J_t)$ is
not solved in general. The main working conjecture here is that such a metric
$\omega_t$ must always be {\it compatible} with the bundle structure (in a
sense made precise in Sect.~\ref{calabi-type} below). As we observe in
Sect.~\ref{partial}, if this conjecture were true it would imply that the
initial bundle $E$ must also split as a direct sum of stable subbundles (and
that $\omega_0$ must be compatible too). We are thus led to believe the
following general statement would be true.

\begin{conj}\label{con:1} A projective bundle $(M,J)=P(E)$ over a compact
curve of genus $\ge 2$ admits an extremal K\"ahler metric in some K\"ahler
class if and only if $E$ decomposes as a direct sum of stable subbundles.
\end{conj}
\begin{rem} {\rm
This conjecture turns out to be true in the case when $E$ is of rank $2$
and $\Sigma$ is a curve of any genus, cf. \cite{AT} for an overview.}
\end{rem}

A partial answer to Conjecture \ref{con:1} is given by the following result
which deals with K\"ahler classes far enough from the boundary of the
K\"ahler cone.

\begin{thm}\label{th:extremal} Let $p\colon P(E) \to \Sigma$ be a holomorphic
projective bundle over a compact complex curve $\Sigma$ of genus $\ge 2$ and
$[\omega_{\Sigma}]$ be a primitive K\"ahler class on $\Sigma$. Then there
exists a $k_0 \in \R$ such that for any $k>k_0$ the K\"ahler class $\Omega_k =
2\pi c_1(\cO(1)_{E}) + k p^*[\omega_{\Sigma}]$ on $(M,J)=P(E)$ admits an
extremal K\"ahler metric if and only if $E$ splits as a direct sum of stable
subbundles.

In the case when $E$ decomposes as the sum of at most two indecomposable
subbundles,\footnote{This is equivalent to requiring that the automorphism
group of $P(E)$ has a maximal torus of dimension $\le 1$.} the conclusion
holds for any K\"ahler class on $P(E)$.
\end{thm}

The proof of Theorem~\ref{th:extremal}, given in Section~\ref{partial},
will be deduced from a general existence theorem established in the much
broader framework of {\it rigid} and {\it semisimple} toric bundles
introduced in \cite{hfkg2}, whose main features are recalled in
Section~\ref{calabi-type} below. As explained in
Remark~\ref{multiplicity-free}, this class of manifolds is closely related
to the class of {\it multiplicity-free} manifolds recently discussed in
Donaldson's paper \cite{Do6}. Our most general existence result can be
stated as follows.

\begin{thm}\label{th:small-classes} Let $(g, \omega)$ be a compatible K\"ahler
metric on $M$, where $M$ is a rigid semisimple toric bundle over a CSC
locally product K\"ahler manifold $(S, g_S, \omega_S)$ with fibres
isomorphic to a toric K\"ahler manifold $(W, \omega_W, g_{W})$, as defined
in Sect.~\textup{3}. Suppose, moreover, that the fibre $W$ admits a {\rm
compatible} extremal K\"ahler metric. Then, for any $k \gg
0$, the K\"ahler class $\Omega_k = [\omega] + kp^*[\omega_S]$ admits
a compatible extremal K\"ahler metric.
\end{thm}

The terms of this statement, in particular the concept of a {\it
compatible} metric, are introduced in Section~\ref{calabi-type}. Its proof,
also given in Section \ref{calabi-type}, uses in a crucial way the
stability under small perturbations of existence of compatible extremal
metric (Proposition~\ref{perturbation}) which constitutes the delicate
technical part of the paper. Another important consequence of
Proposition~\ref{perturbation} is the general openness theorem given by
Corollary~\ref{corollary1}.

\smallskip

A non-trivial assumption in the hypotheses of
Theorem~\ref{th:small-classes} above is the existence of {\it compatible}
extremal K\"ahler metric on the (toric) fibre $W$. This is solved when $W
\cong \C P^{r}$ and $M=P(E)$ with $E$ being holomorphic vector bundle of
rank $r+1$, which is the sum of $\ell+1$ projectively-flat hermitian
bundles, as a consequence of the fact that the Fubini--Study metric on $\C
P^{r}$ admits a non-trivial hamiltonian $2$-form of order $\ell \le r$
(cf.\cite{hfkg1}). We thus derive in Sect.~\ref{sec:proof-existence} the
following existence result.
\begin{thm}\label{th:existence}  Let $p\colon P(E) \to S$ be a holomorphic
projective bundle over a compact K\"ahler manifold $(S, J_S,
\omega_S)$. Suppose that $(S,J_S,\omega_S)$ is covered by the product of
constant scalar curvature K\"ahler manifolds $(S_j,\omega_j), \ j=1, \ldots,
N$, and $E= \bigoplus_{i=0}^{\ell} E_i$ is the direct sum of projectively-flat
hermitian bundles. Suppose further that for each $i$ $c_1(E_i)/{\rm rk}(E_i)-
c_1(E_0)/{\rm rk}(E_0)$ pulls back to $\sum_{j=1}^N p_{ji} [\omega_j]$ on
$\prod_{j=1}^N S_j$ (for some constants $p_{ji}$).  Then there exists a $k_0
\in \R$ such that for any $k>k_0$ the K\"ahler class $\Omega_k = 2\pi
c_1(\cO(1)_{E}) + k p^*[\omega_{S}]$ admits a compatible extremal K\"ahler
metric.
\end{thm}

\begin{rem}\label{rem3} Theorem~\ref{th:existence}  is  closely related to
the results of Y.-J.~Hong in \cite{hong,hong2} who proves, under a
technical assumption on the automorphism group of $S$, that for any
hermitian--Einstein (i.e.~polystable) bundle $E$ over a CSC K\"ahler
manifold $S$, the K\"ahler class $\Omega_k = 2\pi c_1(\cO(1)_{E}) + k
p^*[\omega_{S}]$ for $k\gg 1$ admits a compatible CSC K\"ahler metric if
and only if the corresponding Futaki invariant $\mathfrak{F}_{\Omega_k}$
vanishes. However, in the case when $E$ is not simple (i.e. has
automorphisms other than multiples of identity) the condition
$\mathfrak{F}_{\Omega_k} \equiv 0$ is not in general satisfied for these
classes, see \cite[Sect.~3.4 \& 4.2]{ACGT} for specific examples. Thus,
studying the existence of extremal rather than CSC K\"ahler metrics in
$\Omega_k$ is essential. Another useful remark is that although the
hypothesis in Theorem~\ref{th:existence} that $E$ is the sum of {\it
projectively-flat} hermitian bundles over $S$ is rather restrictive when
$S$ is not a curve, our result strongly suggests that considering $E$ to be
a direct sum of stable bundles (with not necessarily equal slopes) over a
CSC K\"ahler base $S$ would be the right general setting for seeking
extremal K\"ahler metrics in $\Omega_k = 2\pi c_1(\cO(1)_{E}) + k
p^*[\omega_{S}]$ for $k \gg 1$.
\end{rem}

In the final Sect.~\ref{discussion}, we develop further our approach by
extending the leading conjectures~\cite{Do2,Sz} about existence of extremal
K\"ahler metrics on toric varieties to the more general context of {\it
compatible} K\"ahler metrics that we consider in this paper. Thus
motivated, we explore in a greater detail examples when $M$ is a projective
plane bundle over a compact complex curve $\Sigma$. We show that when the
genus of $\Sigma$ is greater than $1$, K\"ahler classes close to the
boundary of the K\"ahler cone of $M$ do not admit any extremal K\"ahler
metric.  In Appendix~\ref{app:almost-kahler}, we introduce the notion of a
{\it compatible extremal almost K\"ahler metric} (the existence of which is
conjecturally equivalent to the existence of a genuine extremal K\"ahler
metric) and show that if the genus of $\Sigma$ is $0$ or $1$, then {\it
any} K\"ahler class on $M$ admits an explicit compatible extremal almost
K\"ahler metric.

\medbreak

The first author was supported in part by an NSERC discovery grant, the
second author by an EPSRC Advanced Research Fellowship and the fourth
author by the Union College Faculty Research Fund.
\section{Proof of Theorem~\ref{main}}\label{s:part1}

As we have already noted in Remark~\ref{rem1}, the `if' part of the
theorem is well-known. So we deal with the `only if' part.

Let $(M,J) = P(E)$, where $\pi\colon E \to \Sigma$ is a holomorphic vector
bundle of rank $m$ over a compact curve $\Sigma$ of genus $\ge 2$.  We want
to prove that $E$ is polystable if $(M,J)$ admits a CSC K\"ahler metric
$\omega$.  Also, by Remark~\ref{rem1}, we will be primarily concerned with
the case when the connected component of the identity ${\rm Aut}_0(M,J)$ of
the automorphisms group of $(M,J)$ is not trivial. Note that, as the normal
bundle to the fibres of $P(E) \to \Sigma$ is trivial and the base is of
genus $\ge 2$, the group ${\rm Aut}_0(M,J)$ reduces to $H^0(\Sigma,
PGL(E))$, the group of fibre-preserving automorphisms of $E$, with Lie
algebra $\mathfrak{h}(M,J) \cong H^0(\Sigma, \mathfrak{sl}(E))$. As any
holomorphic vector field in $\mathfrak{h}(M,J)$ has zeros, the
Lichnerowicz--Matsushima theorem~\cite{lichne,matsushima} implies
$\mathfrak{h}(M,J) = \mathfrak{i}(M,g) \oplus J \mathfrak{i}(M,g),$ where
$\mathfrak{i}(M,g)$ is the Lie algebra of Killing vector fields of $(M,J,
\omega)$. Thus, ${\rm Aut}_0(M,J) \neq \{ {\rm Id} \}$ iff
$\mathfrak{h}(M,J) \neq \{ 0 \}$ iff $\mathfrak{i}(M,g) \neq \{ 0 \}$. We
will fix from now on a maximal torus $\T$ (of dimension $\ell$) in the
connected component of the group of isometries of $(M,J,\omega)$. Note that
$\T$ is a maximal torus in ${\rm Aut}_0(M,J)$ too, by the
Lichnerowicz--Matsushima theorem cited above.

We will complete the proof in three steps, using several lemmas.

\smallskip
We start with following elementary but useful observation which allows us to
relate a maximal torus $\T\subset {\rm Aut}_0(M,J)$ with the structure of $E$.

\begin{lemma}\label{decompose} Let $(M,J)=P(E) \to S$ be a projective bundle
over a compact complex manifold $S$, and suppose that the group
$H^0(S,PGL(E))$ of fibre-preserving automorphisms of $(M,J)$ contains a circle
$S^1$. Then $E$ decomposes as a direct sum $E= \bigoplus_{i=0}^{\ell}E_i$ of
subbundles $E_i$ with $\ell \ge 1$, such that $S^1$ acts on each factor $E_i$
by a scalar multiplication.

In particular, any maximal torus $\T\subset H^0(S,PGL(E))$ arises from a
splitting as above, with $E_i$ indecomposable and $\ell= {\rm dim} (\T )$.
\end{lemma}
\begin{proof} Any $S^1$ in $H^0(S,PGL(E))$ defines a ${\mathbb C}^{\times}$
holomorphic action on $(M,J)$, generated by an element $\Theta \in
\mathfrak{h}(M,J) \cong H^0(S, \mathfrak{sl}(E))$. For any $x\in S$, $\exp{
(t\Theta(x))}, \ t \in \C$ generates a $\C^{\times}$ subgroup of $SL(E_x)$ and
so $\Theta(x)$ must be diagonalizable. The coefficients of the characteristic
polynomial of $\Theta(x)$ are holomorphic functions of $x \in S$, and
therefore are constants. It then follows that $\Theta$ gives rise to a direct
sum decomposition $E= \bigoplus_{i=0}^{\ell} E_i$ where $E_i$ correspond to
the eigenspaces of $\Theta$ at each fibre.

The second part of the lemma follows easily.
\end{proof}

Because of this result and the discussion preceding it, we consider the
decomposition $E = \bigoplus_{i=0}^{\ell} E_i$ as a direct sum of
indecomposable subbundles over $\Sigma$, corresponding to a fixed, maximal
$\ell$-dimensional torus $\T $ in the connected component of the isometry
group of $(g,J,\omega)$. We note that the isometric action of $\T$ is {\it
hamiltonian} as $\T$ has fixed points (on any fibre).

\smallskip
Our second step is to understand the condition that that the Futaki
invariant~\cite{futaki}, with respect to the K\"ahler class
$\Omega=[\omega]$ on $(M,J)$, restricted to the generators of $\T$ is zero.
Hodge theory implies that any (real) holomorphic vector field with zeros on
a compact K\"ahler $2m$-manifold $(M,J, \omega, g)$ can be written as $X=
{\rm grad}_{\omega} f - J {\rm grad}_{\omega} h$, where $f + i h$ is a
complex-valued smooth function on $M$ of zero integral (with respect to the
volume form $\omega^m$), called the {\it holomorphy potential} of $X$, and
where ${\rm grad}_{\omega} f$ stands for the hamiltonian vector field
associated to a smooth function $f$ via $\omega$. Then the (real) Futaki
invariant associates to $X$ the real number
\[
{\mathfrak F}_{\omega} (X) = \int_M f Scal_g \ \omega^m,
\]
where $Scal_g$ is the scalar curvature of $g$. Futaki shows~\cite{futaki} that
${\mathfrak F}_{\omega} (X)$ is independent of the choice of $\omega$ within a
fixed K\"ahler class $\Omega$, and that (trivially) ${\mathfrak F}_{\omega}
(X)=0$ if $\Omega$ contains a CSC K\"ahler metric. A related observation will
be useful to us: with a fixed symplectic form $\omega$, the Futaki invariant
is independent of the choice of compatible almost complex structure within a
path component.

\begin{lemma}\label{symplectic-futaki} Let $J_t$ be a smooth family of
integrable almost-complex structures compatible with a fixed symplectic form
$\omega$, which are invariant under a compact group $G$ of symplectomorphisms
acting in a hamiltonian way on the compact symplectic manifold $(M,\omega)$.
Denote by $\mathfrak {g}_{\omega} \subset C^{\infty}(M)$ the finite
dimensional vector space of smooth functions $f$ such that $X = {\rm
grad}_{\omega} f \in \mathfrak{g}$, where $\mathfrak{g}$ denotes the Lie
algebra of $G$.\footnote{We will tacitly identify throughout the Lie algebra
$\mathfrak g$ of a group $G$ acting effectively on $M$ with the Lie algebra of
vector fields generated by the elements of $\mathfrak{g}$.} Then the
$L^2$-orthogonal projection of the scalar curvature $Scal_{g_t}$ of
$(J_t,\omega, g_t)$ to $\mathfrak{g}_{\omega}$ is independent of $t$.
\end{lemma}
\begin{proof} 
By definition, any $f \in \mathfrak{g}_{\omega}$ defines a
vector field $X= {\rm grad}_{\omega} f$ which is in $\mathfrak{g}$, and is
therefore Killing with respect to any of the K\"ahler metrics $g_t= \omega(
\cdot, J_t \cdot)$. To prove our claim, we have to show that $\int_M f
Scal_{g_t} \omega^m$ is independent of $t$. Using the standard variational
formula for scalar curvature (see e.g.~\cite[Thm.~1.174]{besse}), we compute
\begin{equation}\label{variation}
\frac{d}{dt} Scal_{g_t} = \Delta ({\rm tr}_{g_t} h) + \delta \delta h - g_t(r,
h) = \delta \delta h,
\end{equation}
where $h$ denotes $\frac{d}{dt} g_t$, while $\Delta, \delta$ and $r$ are the
riemannian laplacian, the codifferential and the Ricci tensor of $g_t$,
respectively. Note that to get the last equality, we have used the fact that
$h$ is $J_t$-anti-invariant (as all the $J_t$'s are compatible with $\omega$)
while the metric and the Ricci tensor are $J_t$-invariant (on any K\"ahler
manifold). Integrating against $f$, we obtain
\[
\frac{d}{dt} \Big(\int_M f Scal_{g_t} \omega^m \Big) = \int_M (\delta \delta
h) f\omega^m = \int_M g_t(h, D d f) \omega^m,
\]
where $D$ is the Levi--Civita connection of $g_t$; however, as $f$ is a
Killing potential with respect to the K\"ahler metric $(g_t, J_t)$, it follows
that $Dd f$ is $J_t$-invariant, and therefore $\int_M f Scal_{g_t} \omega^m$
is independent of $t$.
\end{proof}
\begin{rem}\label{rem20} One can extend Lemma~\ref{symplectic-futaki} for
any smooth family of (not necessarily integrable) $G$-invariant almost
complex structures $J_t$ compatible with $\omega$. Then, as shown in
\cite{lejmi}, the $L^2$-projection to $\mathfrak {g}_{\omega}$ of the {\it
hermitian scalar curvature} of the almost K\"ahler metric $(\omega, J_t)$
(see Appendix~\ref{app:almost-kahler} for a precise definition) is
independent of $t$. This gives rise to a {\it symplectic} Futaki invariant
associated to a compact subgroup $G$ of the group of hamiltonian
symplectomorphisms of $(M,\omega)$.
\end{rem}

Lemma~\ref{symplectic-futaki} will be used in conjunction with the
Narasimhan--Ramanan approximation theorem (see~\cite[Prop.~2.6]{NR} and
\cite[Prop.~4.1]{Ram}), which implies that any holomorphic vector bundle
$E$ over a compact curve $\Sigma$ of genus $\ge 2$ can be included in an
analytic family of vector bundles $E_t, \ t \in D_{\varepsilon}$ (where
$D_{\varepsilon} = \{t \in \C, \ |t| < \varepsilon \}$) over $\Sigma$, such
that $E_0:=E$ and $E_t$ is stable for $t\neq 0$. Such a family will be
referred to in the sequel as a {\it small stable deformation} of $E$.

\begin{lemma}\label{futaki-inv} Suppose that the vector bundle
$E= U \oplus V \to \Sigma$ splits as a direct sum of two subbundles, $U$ and
$V$. Consider the holomorphic $S^1$-action on $(M,J)= P(E)$, induced by
fibrewise scalar multiplication by $\exp{ (i \theta)}$ on $V$, and let $X \in
\mathfrak{h}(M,J)$ be the (real) holomorphic vector field generating this
action. Then the Futaki invariant of $X$ with respect to some (and therefore
any) K\"ahler class $\Omega$ on $(M,J)$ vanishes if and only if $U$ and $V$
have the same slope.
\end{lemma}
\begin{proof} We take some K\"ahler form $\omega$ on $(M,J)=P(E)$ and, by
averaging it over $S^1$, we assume that $\omega$ is $S^1$-invariant. As the
$S^1$-action has fixed points, the corresponding real vector field $X$ is
$J$-holomorphic and $\omega$-hamiltonian, i.e., $X= {\rm grad}_{\omega} f$ for
some smooth function $f$ with $\int_M f \omega^m =0$.

We now consider small stable deformations $U_t, V_t, \ t \in D_{\varepsilon}$
of $U$ and $V$, and put $E_t = U_t \oplus V_t$.  Considering the projective
bundle $P(E_t)$, we obtain a non-singular Kuranishi family $(M,J_t)$ with
$J_0=J$. By the Kodaira stability theorem (see e.g.~\cite{KM}) one can find a
smooth family of K\"ahler metrics $(\omega_t, J_t)$ with $\omega_0 =
\omega$. Using the vanishing of the Dolbeault groups $H^{2,0}(M,J_t)=
H^{0,2}(M,J_t)=0$, Hodge theory implies that by decreasing the initial
$\varepsilon $ if necessary, we can assume $[\omega_t] = [\omega]$ in
$H^2_{dR}(M)$. Note that any $J_t$ is $S^1$-invariant so, by averaging over
$S^1$, we can also assume that $\omega_t$ is $S^1$-invariant. Applying the
equivariant Moser lemma, one can find $S^1$-equivariant diffeomorphisms,
$\Phi_t$, such that $\Phi_t^* \omega_t = \omega$. Considering the pullback of
$J_t$ by $\Phi_t$, the upshot from this construction is that we have found a
smooth family of integrable complex structures $J_t$ such that: (1) each $J_t$
is compatible with the fixed symplectic form $\omega$ and is $S^1$-invariant;
(2) $J_0=J$; (3) for $t\neq 0$, the complex manifold $(M,J_t)$ is
equivariantly biholomorphic to $P(U_t \oplus V_t) \to \Sigma$ with $U_t$ and
$V_t$ stable (and therefore projectively-flat) hermitian bundles.

If $U$ and $V$ have equal slopes, then $E_t=U_t \oplus V_t$ becomes polystable
for $t \neq 0$, and $(M,J_t)$ has a CSC K\"ahler metric in each K\"ahler
class. It follows that the Futaki invariant of $X$ on $(M, J_t, \omega)$ is
zero for $t \neq 0$.

Conversely, if $U$ and $V$ have different slopes, it is shown in
\cite[Sect.~3.2]{ACGT} that the Futaki invariant of $X$ is different from zero
for any K\"ahler class on $(M,J_t), \ t \neq 0$.

We conclude using Lemma~\ref{symplectic-futaki}.
\end{proof}

This lemma shows that all the factors in the decomposition $E =
\bigoplus_{i=0}^{\ell} E_i$ must have equal slope, should a CSC K\"ahler
metric exists. As in the proof of Lemma~\ref{futaki-inv}, we consider small
stable deformations $E_i(t)$ of $E_i$ and our assumption for the slopes
insures that $E(t) = \bigoplus_{i=0}^{\ell} E_i(t)$ is polystable for $t
\neq 0$; furthermore, by acting with $\T$-equivariant diffeomorphisms, we
obtain a smooth family of $\T$-invariant complex structures $J_t$
compatible with $\omega$, such that for $t\neq 0$, the complex manifold
$(M, J_t)$ has a locally-symmetric CSC K\"ahler metric in each K\"ahler
class; by the uniqueness of the extremal K\"ahler metrics modulo
automorphisms~\cite{CT, mab-two}, {\it any} extremal K\"ahler metric on
$(M,J_t)$ is locally-symmetric when $t\neq 0$.  The third step in the proof
of Theorem~\ref{main} is then to show that the initial CSC K\"ahler metric
$(J_0,\omega)$ must be locally symmetric too. This follows from the next
technical result, generalizing arguments of \cite{Le-Sim1,Fuj-Sch}.

\begin{lemma}\label{stability} Let $J_t$ be a smooth family of integrable
almost-complex structures compatible with a symplectic form $\omega$ on a
compact manifold $M$, which are invariant under a torus $\T$ of hamiltonian
symplectomorphisms of $(M,\omega)$. Suppose, moreover, that $(J_0,\omega)$
define an extremal K\"ahler metric and that $\T$ is a maximal torus in the
reduced automorphism group of $(M,J_0)$. Then there exists a smooth family
of extremal K\"ahler metrics $(J_t,\omega_t, g_t)$, defined for
sufficiently small $t$, such that $\omega_0= \omega$ and $[\omega_t] =
[\omega]$ in $H^2_{dR}(M)$.
\end{lemma} 
\begin{proof} Recall that~\cite{gauduchon-book,Le-Sim} on any compact K\"ahler
manifold $(M,J)$, the {\it reduced} automorphism group, ${\widetilde {\rm
Aut}}_0(M,J)$, is the identity component of the kernel of the natural group
homomorphism from ${\rm Aut}_0(M,J)$ to the {\it Albanese torus} of $(M,J)$;
it is also the connected closed subgroup of ${\rm Aut}_0(M,J)$, whose Lie
algebra $\mathfrak{h}_0(M,J) \subset \mathfrak{h}(M,J)$ is the ideal of
holomorphic vector fields with zeros.

We denote by $\mathfrak{t}$ the Lie algebra of $\T$ and by $\mathfrak{h}$
(resp. $\mathfrak{h}_0$) the Lie algebra of the complex automorphism group
(resp. {reduced} automorphism group) of the central fibre $(M,J_0)$. As
$\T$ acts in a hamiltonian way, we have $\mathfrak{t} \subset
\mathfrak{h}_0$. By assumption, $\mathfrak{t}$ is a maximal abelian
subalgebra of $\mathfrak{i}_0(M,g_0)=\mathfrak{i}(M,g_0)
\cap\mathfrak{h}_0$, where $\mathfrak{i}(M,g_0)$ is the Lie algebra of
Killing vector fields of $(M, J_0, \omega, g_0)$.

As in the Lemma~\ref{symplectic-futaki} above, we let $\mathfrak{t}_{\omega}
\subset C^{\infty}(M)$ be the finite dimensional space of smooth functions
which are hamiltonians of elements of $\mathfrak{t}$. As the K\"ahler metric
$(J_0,\omega, g_0)$ is extremal (by assumption), its scalar curvature
$Scal_{g_0}$ is hamiltonian of a Killing vector field $X= {\rm grad}_{\omega}
(Scal_{g_0}) \in \mathfrak{i}_{0}(M,g_0)$. Clearly, such a vector field is
central, so $X \in \mathfrak{t}$ (by the maximality of $\mathfrak{t}$) and
therefore $Scal_{g_0} \in \mathfrak{t}_{\omega}$.

For any $\T$-invariant K\"ahler metric $({\tilde J}, {\tilde \omega}, {\tilde
g})$ on $M$, we denote by $\mathfrak{t}_{\tilde \omega}$ the corresponding
space of Killing potentials of elements of $\mathfrak{t}$ (noting that any $X
\in \mathfrak{t}$ has zeros, so that $\T$ belongs to the reduced automorphism
group of $(M, {\tilde J})$), and by $\Pi_{\tilde \omega}$ the $L^2$-orthogonal
projection of smooth function to $\mathfrak{t}_{\tilde \omega}$, with respect
to the volume form ${\tilde \omega}^m$. Obviously, if the scalar curvature
$Scal_{\tilde g}$ of ${\tilde g}$ belongs to $\mathfrak{t}_{\tilde \omega}$,
then $\tilde g$ is extremal.

Following \cite{Le-Sim1}, let ${C_{\perp}^{\infty}(M)}^{\T}$ denote the
Fr\'echet space of $\T$-invariant smooth functions on $M$, which are
$L^2$-orthogonal (with respect to the volume form $\omega^m$) to
$\mathfrak{t}_{\omega}$, and let ${\mathcal U}$ be an open set in $\R \times
{C_{\perp}^{\infty}(M)}^{\T}$ of elements $(t, f)$ such that $\omega +
dd^c_t f$ is K\"ahler with respect to $J_t$ (here $d^c_t$ denotes the
$d^c$-differential corresponding to $J_t$).  We then consider the map $\Psi :
{\mathcal U} \mapsto \R \times {C_{\perp}^{\infty}(M)}^{\T}, $ defined by
\[
\Psi(t, f) = \Big(t,({\rm Id} - \Pi_{\omega}) \circ ({\rm Id} - \Pi_{\tilde
\omega}) (Scal_{\tilde g})\Big),
\]
where $\tilde \omega := \omega + dd^c_t f$ and $Scal_{\tilde g}$ is the scalar
curvature of the K\"ahler metric $\tilde g$ defined by $(J_t, {\tilde
\omega})$. One can check that this map is $C^1$ and compute (as in
\cite{Le-Sim}, by also using \eqref{variation}) that its differential at
$(0,0) \in {\mathcal U}$ is
\[
(T_{(0,0)}\Psi) (t, f) = (t, t \delta \delta h -2 \delta \delta (D df)^-),
\]
where $D$ and $\delta$ are respectively the Levi--Civita connection and the
codifferential of $g_0$, $h= \big(\frac{d g_t}{dt}\big)_{t=0}$ and $(Ddf)^-$
denotes the $J_0$-anti-invariant part of $Ddf$. Note that $L(f) :=\delta
\delta ((D df)^-)$ is a $4$-th order (formally) self-adjoint $\T$-invariant
elliptic linear operator (known also as the {\it Lichnerowicz} operator, see
e.g.~\cite{gauduchon-book}). When acting on smooth functions, $L$ annihilates
$\mathfrak{t}_{\omega}$ (because any Killing potential $f$ satisfies $(D
df)^-=0$). It then follows that $L$ leaves ${C_{\perp}^{\infty}(M)}^{\T}$
invariant and, by standard elliptic theory, we obtain an $L^2$-orthogonal
splitting ${C_{\perp}^{\infty}(M)}^{\T} = {\rm ker}(L) \oplus {\rm
im}(L)$. However, any smooth $\T$-invariant function $f$ in ${\rm ker}(L)$
gives rise to a Killing field $X = {\rm grad}_{\omega} f$ in the centralizer
of $\mathfrak t \subset \mathfrak{i}_0(M,g_0)$. As $\mathfrak{t}$ is a maximal
abelian subalgebra of $\mathfrak{i}_0(M,g_0)$ we must have $X \in \mathfrak
t$, i.e. $f \in \mathfrak{t}_{\omega}$. It follows that the kernel of $L$
restricted to ${C_{\perp}^{\infty}(M)}^{\T}$ is trivial, and therefore $L$
is an isomorphism of the Fr\'echet space ${C_{\perp}^{\infty}(M)}^{\T}$.

This understood, we are in position to apply standard arguments, using the
implicit function theorem for the extension of $\Psi$ to the Sobolev
completion ${L_{\perp}^{2,k}(M)}^{\T}$ (with $ k\gg 1$) of
${C_{\perp}^{\infty}(M)}^{\T}$, together with the regularity result for
extremal K\"ahler metrics, precisely as in
\cite{Le-Sim,Le-Sim1,Fuj-Sch}. We thus obtain a family $(t, \omega_t)$ of
smooth, $\T$-invariant extremal K\"ahler metrics $(J_t, \omega_t)$ (defined
for $t$ in a small interval about $0$) which converge to the initial
extremal K\"ahler metric $(J_0, \omega)$ (in any Sobolev space $L^{2,k}(M),
\ k\gg 1$, and hence, by the Sobolev embedding, in $C^{\infty}(M)$).
\end{proof}

The uniqueness argument thus also applies at $t=0$, and the initial metric
is locally symmetric. We can now conclude the proof of Theorem~\ref{main}
by a standard argument using the de Rham decomposition theorem (see
\cite[Lemma 8]{fujiki} and \cite{LeB}). This realizes the fundamental group
of $\Sigma$ as a discrete subgroup group of isometries of the hermitian
symmetric space ${\C} P^{m-1}\times {\mathbb H}$ and thus defines a
projectively flat structure on $P(E)\to \Sigma$.

\section{Rigid toric bundles and the generalized Calabi construction} \label{calabi-type}

In this section, we recall the notion of a {\it semisimple} and {\it rigid}
isometric hamiltonian action of a torus on a compact K\"ahler manifold
$(M,g,J,\omega)$ (introduced in \cite{hfkg2}), as well as the construction
of {\it compatible} K\"ahler metrics (given by the {\it generalized Calabi
construction} of \cite{hfkg2}) on such manifolds. This provides a framework
for the search of extremal {\it compatible} metrics on rigid toric bundles
over a semisimple base, which parallels (and extends) the theory of
extremal toric metrics developed in \cite{Do2,Do3,Do5}. We then apply the
construction of this section to projective bundles of the form $P(E_0
\oplus \cdots \oplus E_{\ell}) \to S$, where $E_i$ is a projectively-flat
hermitian bundle over a K\"ahler manifold $(S, \omega_S)$. In all cases, we
prove the existence of compatible extremal K\"ahler metrics in ``small''
K\"ahler classes, cf. Theorems~\ref{th:small-classes} and
\ref{th:existence}.

\subsection{Rigid torus actions} \label{s:rigid}
Most of the material in this section is taken from \cite[Sect.~2]{hfkg2}
and we refer the Reader to this article for further details.

\begin{defn}\label{rigid} Let $(M,g,J,\omega)$ be a connected K\"ahler
$2m$-manifold with an effective isometric hamiltonian action of an
$\ell$-torus $\T$ with momentum map $z \colon M\to{\mathfrak t}^*$.  We say
the action is {\it rigid} if for all $x\in M$, $R_x^*g$ depends only on
$z(x)$, where $R_x\colon \T\to \T\cdot x\subset M$ is the orbit map.
\end{defn}

In other words, the action is rigid if, for any two generators
$X_\xi,X_\eta$ of the action --- $\xi,\eta\in\mathfrak t$ --- the smooth
function $g(X_\xi,X_\eta)$ is constant on the levels of the momentum map
$z$.

\smallskip
Henceforth, we suppose  that $M$ is compact. 

\smallskip

Obvious and well-known examples of rigid toric actions are provided by 
toric K\"ahler manifolds. 
A key feature of toric K\"ahler manifolds is actually  shared by 
rigid torus actions, namely the fact  that the image of $M$ by the momentum
map is a {\it Delzant} polytope $\Delta \subset {\mathfrak t}^*$ (see
\cite[Prop.~4]{hfkg2}) and that the regular values of $z$ are the points in
the interior $\Delta^0$. Thus, to any compact K\"ahler manifold endowed with a
rigid isometric hamiltonian action of an $\ell$-torus $\T$, one can associate
a smooth compact toric symplectic $2\ell$-manifold $(V, \omega_V, \T)$, via
the Delzant correspondence~\cite{delzant}. Note that the Delzant construction
also endows $V$ with the structure of a complex toric variety $(V, J_V,
\T^c)$.

Another {\it smooth} variety  is associated to a rigid torus action, namely the
complex --- or stable ---  quotient $\hat S$ of $(M,J)$ by the 
complexified action of
$\T^c$. For a general torus action, $\hat S$ is a $2(m-\ell)$-dimensional complex orbifold, but
when the torus action is rigid, 
it is shown in~\cite[Prop.~5]{hfkg2} that $\hat
S$ is smooth, and $M^0:=z^{-1}(\Delta^0)$ is then 
a principal $\T^c$ bundle over
$\hat S$.  Denote by  $\hat{M} :=  M^0 \times_{\T^c} V \to \hat S$
the associated fibre bundle in toric manifolds. Then, {\it either
  $(M,J)$ is 
$\T^c$-{\it equivariantly} biholomorphic to $\hat{M}$ or it is obtained 
by {\rm (}$\T^c$-equivariantly{\rm )} {\em blowing down} the
inverse image in $\hat{M}$ of 
 some codimension one faces of $\Delta$.}  Thus, $M$ and
$\hat M$ are (different in general) $\T^c$-equivariant compactifications of
the same principal $\T^c$-bundle $M^0$ over $\hat S$.  

In either case, by a convenient abuse of notation, we call $M$ or any
complex manifold $\T^c$-equivariantly biholomorphic to $M$, a {\it rigid
toric bundle}.  In the case when there is no blow-down, then $M = \hat M$
is a genuine fibre bundle over $\hat S$ with fibre the toric manifold $V$,
associated to a principal $\T$-bundle over $\hat S$, whereas, in the
general case, the K\"ahler metric $g$ on $M$ will be described, via its
pullback on $\hat{M}$, in terms of the toric bundle structure of
$\hat{M}$, thus allowing to introduce the notion of {\it compatible}
K\"ahler metrics on a general rigid toric bundle,
cf. \S\ref{s:generalised-calabi}.

We now specialize the above construction, in particular the blow-down
procedure, in the case when the (rigid) torus action is, in addition, {\it
semisimple}, according to the following general definition.

\begin{defn} An isometric hamiltonian torus action on K\"ahler manifold
$(M,g, J, \omega)$ is {\it semisimple} if for any regular value $z_0$ of the
momentum map, the derivative with respect to $z$ of the family $\omega_{\hat
S}(z)$ of K\"ahler forms on the complex (stable) quotient $\hat S$ of $(M,J)$
(induced by the symplectic quotient construction at $z$) is parallel and
diagonalizable with respect to $\omega_{\hat S}(z_0)$.\footnote{In general,
$\hat S$ is well-defined as a complex orbifold for $z$ in the connected
component $U_{z_0}$ of $z_0$ in the regular values.}\end{defn}

For a {\it semisimple} and {\it rigid} isometric hamiltonian torus action
the K\"ahler metrics $\omega_{\hat S}(z)$, parametrized by $z$ in $\Delta
^0$, on the stable quotient $\hat{S}$ are simultaneously diagonal and have
the same Levi--Civita connection. There then exists a K\"ahler metric
$(g_{\hat S}, \omega_{\hat S})$ on ${\hat S}$, such that the K\"ahler forms
$\omega_{\hat S}(z)$ are simultaneously diagonalizable with respect to
$g_{\hat S}$ and parallel with respect to the Levi--Civita connection of
$g_{\hat S}$, so that the universal cover of $(\hat S, \omega_{\hat S})$ is
a product $\prod _{j = 1} ^N (S _j, \omega _j)$ of K\"ahler manifolds
$(S_j, \omega_j)$ of dimensions $2d_j$, $j= 1, \ldots, N$, in such a way
that the restriction to $S_j$ of the pullback of $\omega_{\hat S}(z)$ is a
multiple of $\omega _j$ by an affine function of $z$. Moreover, to any face
of codimension one, $F _b$, of $\Delta$ involved in the blow-down process
corresponds a factor $(S _b, \omega _b)$ in the product $\prod _{j = 1} ^N
(S _j, \omega _j)$, which is isomorphic to the standard complex projective
space $(\mathbb{C} P ^{d _b}, \omega _b)$ of (positive) complex dimension
$d _b$ equipped with a Fubini-Study metric of holomorphic sectional
curvature equal to $2$ --- equivalently of scalar curvature equal to $2 d
_b (1 + d _b)$ --- so that $[\omega _b] = 2 \pi c _1 (\mathcal{O}
_{\mathbb{C} P ^{d _b}} (1))$.

Conversely, let $\hat S$ be a compact K\"ahler manifold, whose universal
cover is a K\"ahler product of the form $\prod _{j = 1} ^N (S _j, \omega
_j) = \prod _{a \in \mathcal{A}} (S _a, \omega _a) \times \prod _{b \in
\mathcal{B}} (\mathbb{C} P ^{d _b}, \omega _b)$, where each $\omega _b$ is
the K\"ahler form of a Fubini-Study metric of holomorphic sectional
curvature equal to $2$ ($\mathcal{A}$ or $\mathcal{B}$ may possibly be
empty). We moreover assume that $\pi _1 (\hat{S})$ acts diagonally by
K\"ahler isometries on the universal cover, so that $\hat{S}$ has the
structure of a fibre product of flat unitary $\mathbb{C} P ^{d
_b}$-bundles, $b \in \mathcal{B}$, over a compact K\"ahler manifold $S$,
covered by the product $\prod _{a \in \mathcal{A}} (S _a, \omega _a)$. Let
$\mathbb{T}$ a real (compact) torus of dimension $\ell$, of Lie algebra
$\mathfrak{t}$, $\Delta$ be a Delzant polytope in the dual space
$\mathfrak{t} ^*$, and $(V, J _V, \omega _V, \mathbb{T})$ a
$\mathbb{T}$-toric K\"ahler $2 \ell$-manifold, with momentum polytope
$\Delta$. Among the $n$ codimension one faces $F _i$, $i = 1, \ldots, n$,
of $\Delta$, with inward normals $u _i$ in $\mathfrak{t}$, we distinguish a
subset $\{F _b:b \in \mathcal{B}\}$ (possibly empty) with inward normals $u
_b$.  Let $\hat{P}$ be a principal $\mathbb{T}$-bundle over $\hat{S}$, such
that $- 2 \pi c _1 (\hat{P})$, as a $\mathfrak{t}$-valued $2$-form, is {\it
diagonalizable} with respect to the local product structure of $\hat{S}$,
i.e. is of the form $\sum _{j = 1} ^N [\omega _j] \otimes p _j = \sum _{a
\in \mathcal{A}} [\omega _a] \otimes p _a + \sum _{b \in \mathcal{B}}
[\omega _b] \otimes u _b$, where all $p _j$ are (constant) elements of
$\mathbb{T}$ and, we recall, $u _b$ denotes the inward normal of the
distinguished codimension one face of $\Delta$ associated to the factor
$(\mathbb{C} P ^{d _b}, \omega _b)$ in the universal cover of $\hat{S}$. We
denote by $\hat{M} = \hat{P} \times _{\mathbb{T}} V$ the associated toric
bundle over $\hat{S}$.

With these data in hand, the blow-down process relies on the general
{\it restricted toric quotient construction}, introduced in our previous
work \cite{hfkg2}, which, in the current situation, goes as follows.

Consider the product manifold $S_0 = \prod_{b \in \cB} \C P^{d_b}$ equipped
with a principal $\T$-bundle $P_0$ with $c_1(P_0)= \sum_{b \in \cB} c_1(\cO
_{\mathbb{C} P ^{d _b}} (-1))\otimes u_b$, and the corresponding bundle of
toric K\"ahler manifolds $\hat W = P_0\times_{\T} V$ over $S _0$, with
momentum map $z\colon \hat{W} \to \Delta \subset \mathfrak{t} ^*$.  Then the
restricted toric quotient construction associates to $\hat W$ a toric
manifold $(W, J_{W}, T)$, of the same dimension $2 (\ell + \sum _{b \in
\mathcal{B}} d _b)$ as $\hat{W}$, obtained from $\hat W$ by collapsing $z
^{-1} (F _b)$, $b \in \mathcal{B}$. Recall that, whereas $V$ is obtained,
via the Delzant construction, as a symplectic reduction of $\mathbb{C} ^n$
by the $(n - \ell)$-dimensional torus $G$, kernel of the map $(a _1,
\ldots, a _n) \mod{\mathbb{Z} ^n} \mapsto \sum _{i = 1} ^n a _i u _i
\mod{\Lambda}$ from $\mathbb{T} ^n = \mathbb{R} ^n/\mathbb{Z} ^n$ onto
$\mathbb{T}$, $W$ is similarly obtained as a symplectic reduction of
$\oplus _{b \in \mathbb{B}} \mathbb{C} ^{d _b + 1} \oplus \mathbb{C} ^{n -
|\mathcal{B}|}$ by $G \subset \mathbb{T} ^n$, via the natural diagonal
action of $\mathbb{T} ^n$ on $\oplus _{b \in \mathbb{B}} \mathbb{C} ^{d _b
+ 1} \oplus \mathbb{C} ^{n - |\mathcal{B}|}$ (where $|\mathcal{B}|$ is the
cardinality of $\mathcal{B}$); in this picture, the $(\ell + \sum _{b \in
\mathcal{B}} d _b)$-dimensional torus $T$ acting on $W$ is identified with
the quotient $\mathbb{T} ^{n + \sum _{b \in \mathcal{B}} d _b}/G$, whereas
the {\it restricted subtorus} $\mathbb{T}$ is identified with the subtorus
$\mathbb{T} ^n/G$ of $T$, cf.~\cite[Sect.~1.6]{hfkg2} for details, in
particular for the identification of $W$ with a blow-down of
$\hat{W}$.\footnote{A simple illustration of this construction is $W=\C
P^2$ seen as a fibrewise $S^1$-equivariant blow-down of the first
Hirzebruch surface $\hat W=P(\cO \oplus \cO(-1)) \to \C P^1$.} We denote by
$b\colon \hat W \to W$ the $(\mathbb{T}, T)$-equivariant blow-down map of $\hat
W$ onto $W$, along the inclusion $\T \subset T$. We then have the following
definition.
\begin{defn}\label{blow-down} A blow-down of
$\hat M = \hat{P}\times_{ \T} V\to \hat S \to S$ is a (locally trivial)
fibre bundle $(M, J)$ over $S$, with fibre the K\"ahler manifold $(W,
J_{W})$, endowed with a global fibrewise (restricted) holomorphic action of
the $\ell$-dimensional torus $\T$, and a $\T$-equivariant holomorphic map
$\hat M \to M$, equal to the blow-down map $b$ on the corresponding
K\"ahler fibres over $S$. We summarize this definition in a
$\T$-equivariant commutative diagram
\begin{diagram}[size=1.5em]
\smash{\hat M}= \hat P \times_{\T} V &\rTo& \smash{\hat S}\\
\dTo & & \dTo\\
M & \rTo & S,
\end{diagram}
and refer to the manifold $(M,J, \T)$ as a {\it rigid toric bundle over a
semisimple base}.
\end{defn}

Using this construction, the blow-down was introduced in \cite{hfkg2} under
the simplifying assumption that the local product structure of $\hat S$
consists of {\it global} factors for $b \in \cB$ (i.e. $\hat S \to S$ is a trivial fibre bundle). In particular, the blow-down
was expressed in \cite[Sect.~2.5]{hfkg2} in terms of the universal covers of
$M$ and $\hat M$. In fact, in this case there exists a
diagonalizable principal $\T$-bundle $P$ over $S$ with first Chern class $2\pi
c_1(P)= \sum_{a \in \cA} [\omega_a]\otimes p_a$ and we can identify $\hat M =  \hat P \times_{\T} V \cong P \times_{\T} {\hat W}$. Then,  $M:= P \times_{\T} W$ clearly satisfies the definition~\ref{blow-down} above. 

\smallskip

We now illustrate the blow-down construction in the case of projective bundles.

\subsection{Projective bundles as rigid toric bundles}
\label{s:projective bundles}

In this paragraph, we specialize the previous discussion 
to the case when the Delzant 
polytope $\Delta$ is a simplex
in $\mathfrak t^*\cong\R^{\ell}$,  with codimension one faces $F_0,\ldots
F_\ell$; the associated complex toric variety $V$ is then the complex
projective space $V\cong (\C P^{\ell}, \T^c)$ and $\smash{\hat M}$ is 
then $\T^c$-equivariantly
biholomorphic to a $\C P^{\ell}$-bundle over a K\"ahler manifold $\hat
S$ of the type discussed in \S\ref{s:rigid}; 
since $\smash{\hat M}$ comes from a {\it principal} $\T^c$-bundle,
$\smash{\hat M}$ is actually of the form 
$P(\cL_0\oplus \cdots \oplus \cL_\ell)\to \hat S,$ where
$\cL_i$ are hermitian holomorphic line bundles (the $\T^c$ action is
then induced by 
scalar multiplication on $\cL_i$).

According to the discussion in \S\ref{s:rigid},  a blow-down
process on $\hat{M}$ is encoded
by the realization of $\hat S$ as a fibre product of {\it flat projective
unitary} $\C P^{d_{b}}$-bundles over a K\"ahler manifold $S$. 
We here only consider flat projective bundles of the form $P
(E)$, where $E$ is a rank $r+1$ projectively-flat
hermitian vector bundle over  $S$ (in general the obstruction to the existence of $E$ is given by a torsion
element of $H^2(S, \cO^*)$; in particular,
such an $E$ always exists if $S=\Sigma$ is a Riemann surface).
We then have $\hat S = P(\vE_0)\times_S \cdots\times_S P(\vE_\ell)\to
S$, where each $\vE_i\to S$ is a projectively-flat hermitian bundle of rank
$d_i+1$, and we assume that 
 $c_1(E_i)/(d_i+1) - c_1(E_0)/(d_0+1)$ pulls back to $\sum_{a
\in \cA} p_{ia}[\omega_a]$ on the covering space $\prod_{a\in \cA}(S_a,
\omega_a)$.  

In this case, we have that $\smash{\hat M}=P\bigl(\cO(-1)_{E_0}\oplus \cdots
\oplus\cO(-1)_{E_\ell}\bigr)\to \hat S$, where $\cO(-1)_{E_i}$ is the
(fibrewise) tautological line bundle over $P(\vE_i) \to S$ --- trivial over the
other factors of $\hat S$ ---  whereas $M = P\bigl(E _0 \oplus \cdots
\oplus E _{\ell}\bigr) \to S$, the blow-down process being, over each
point of $S$, the standard blow-down process from $P \bigl(\oplus _{j = 0}
^{\ell} \mathcal{O} (-1) _{V _j}\bigr) \to \prod _{j = 0} ^{\ell} (V
_j)$ to $P (V)$, for any splitting $V = \oplus _{j = 0} ^{\ell} V _j$
of a complex vector space $V$ into a direct sum of $\ell + 1$ $(d _j +
1)$-dimensional vector subspaces, $d _j > 0$, $\ell > 0$,
cf. \cite{hfkg2}

To go further into the geometry of the situation, we next fix a hermitian
metric on $E_i$ whose Chern connection has curvature $\Omega_i \otimes {\rm
Id}_{E_i}$ with
\begin{equation*}
\Omega_i - \Omega_0 = \sum_{a \in \cA} p_{ia} \omega_a, \ \ i\ge 1,
\end{equation*}
where $p_a=(p_{1a}, \ldots , p_{\ell a}) \in \R^{\ell} \cong \mathfrak t$
will be the constants of our construction.  Let ${\tilde \theta}_{i}$ be a
connection $1$-form for the principal $U(1)$-bundle over $\hat S$,
associated to the line bundle $\cO(-1)_{E_i}$, with curvature $d\tilde
\theta_i = -\omega_i + \Omega_i$, where $\omega_i$ pulls back to the
Fubini--Study metric of scalar curvature $2d_i(d_i+1)$ on the universal
cover of $P(E_i)$ when $d_i \ge 1$, and is zero when $d_i=0$. We then put
${\hat \theta}_j = {\tilde \theta}_j - {\tilde \theta_0}$ to define a
principal $\T$-connection ${\hat \theta} = ({\hat \theta}_1, \ldots, {\hat
\theta}_{\ell})$ associated with the principal $\T^c$-bundle $M^0$ over
$\hat S$.

\subsection{The generalized Calabi construction on rigid toric bundles over a
semi\-simple base}\label{s:generalised-calabi}

As recalled in \S\ref{s:rigid}, any compact K\"ahler manifold
$(M,J,\omega,g)$ endowed with a rigid and semisimple isometric hamiltonian
action of an $\ell$-torus $\T$, is equivariantly biholomorphic to a {\it
rigid toric bundle} over a semisimple base, obtained by a blow-down process
from an associated bundle in $\T$-toric manifolds $\hat{M}$.  It still
remains to describe K\"ahler structure $(g, \omega)$ on $M$: according to
\cite[Thm.~2]{hfkg2}, this is done by using the {\it generalized Calabi
construction} which we now recall, following \cite{hfkg2}, with slightly
different notation. We freely use the notation of \S\ref{s:rigid}.

The generalized Calabi construction is made of three main building blocks
--- only two if there is no blow-down --- and produces a family of (smooth)
singular K\"ahler structures on $\hat{M}$, which descend to genuine
K\"ahler metrics on $M$, called {\it compatible}: for any K\"ahler manifold
$(M,J,\omega,g)$ endowed with a rigid and semisimple isometric hamiltonian
action of an $\ell$-torus $\T$, the K\"ahler structures $(g,\omega)$ is
compatible.

\smallskip
The first building block of the construction is the choice of a compatible
$\T$-invariant K\"ahler metric $g_V$ on the symplectic toric manifold $(V,
\omega_V, \T)$. This part is well-known (see
e.g.~\cite{Abreu0,Abreu1,Do3,Guillemin}): let $z\in C^\infty(V,\mathfrak
t^*)$ be the momentum map of the $\T$ action with image $\Delta$ and let
$V^0 = z^{-1}(\Delta^0)$ be the union of the generic $\T$ orbits. On $V^0$,
orthogonal to the $\T$ orbits is a rank $\ell$ distribution spanned by
commuting holomorphic vector fields $JX_\xi$ for $\xi\in\mathfrak t$. Hence
there is a function $\ang\colon V^0\to \mathfrak t/2\pi \Lambda$, defined
up to an additive constant, such that $d\ang(JX_\xi)=0$ and
$d\ang(X_\xi)=\xi$ for $\xi\in\mathfrak t$. The components of $\ang$ are
`angular variables', complementary to the components of the momentum map
$z\colon V^0\to \mathfrak t^*$, and the symplectic form in these
coordinates is simply
\begin{equation*}\label{toricomega}
\omega_V=\langle dz \wedge d\ang\rangle,
\end{equation*}
where the angle brackets denote contraction of $\mathfrak t$ and $\mathfrak
t^*$.
These coordinates identify each tangent space with $\mathfrak t\oplus
\mathfrak t^*$, so any $\T$-invariant $\omega_V$-compatible K\"ahler metric
must be of the form
\begin{equation}\label{toricmetric}
g_V=\langle dz, {\bf G} , dz \rangle+ \langle d\ang,
{\bf H}, d\ang\rangle,
\end{equation}
where ${\bf G}$ is a positive definite $S^2\mathfrak t$-valued function on
$\Delta^0$, ${\bf H}$ is its inverse in $S^2\mathfrak t^*$---observe that
${\bf G}$ and ${\bf H}$ define mutually inverse linear maps $\mathfrak t^*
\to\mathfrak t$ and $\mathfrak t\to\mathfrak t^*$ at each point---and
$\langle\cdot,\cdot,\cdot\rangle$ denotes the pointwise contraction
${\mathfrak t}^* \times S^2\mathfrak t \times {\mathfrak t}^* \to \R$ or the
dual contraction.  The corresponding almost complex structure is defined by
\begin{equation}\label{toricJ}
J d\ang = -\langle {\bf G}, dz \rangle
\end{equation}
from which it follows that $J$ is integrable if and only if ${\bf G}$ is the
hessian of a function $U$ (called symplectic potential) on
$\Delta^0$~\cite{Guillemin}.

Necessary and sufficient conditions for $U$ to come from a globally defined
$\T$-invariant $\omega_V$-compatible K\"ahler metric on $V$ were  obtained in
\cite{Abreu1,hfkg2,Do3}. We state here the first-order boundary conditions
obtained in \cite[Prop.~1]{hfkg2}: for any face $F \subset \Delta$, denote by
${\mathfrak t}_{F} \subset {\mathfrak t}$ the vector subspace spanned by the
inward normals $u_i \in {\mathfrak t}$ to all codimension one faces of
$\Delta$, containing $F$; as $\Delta$ is Delzant, the codimension of
${\mathfrak t}_{F}$ equals the dimension of $F$. Furthermore, the annihilator
${\mathfrak t}^0_{F}$ of ${\mathfrak t}_{F}$ in ${\mathfrak t}^*$ is naturally
identified with $({\mathfrak t}/{\mathfrak t}_F)^*$. Then {\it a smooth
strictly convex function $U$ on $\Delta^0$ corresponds to a $\T$-invariant,
$\omega_V$-compatible K\"ahler metric $g_V$ via \eqref{toricmetric} if and
only if the $S^2{\mathfrak t}^*$-valued function ${\bf H}= {\rm Hess}
(U)^{-1}$ on $\Delta^0$ verifies the following boundary conditions}:
\begin{bulletlist}
\item \textup{[smoothness]} ${\bf H}$ is the restriction to $\Delta^0$ of a
smooth $S^2{\mathfrak t}^*$-valued function on $\Delta$\textup;
\item \textup{[boundary values]} for any point $z$ on the codimension one face
$F_i \subset \Delta$ with inward normal $u_i$, we have
\begin{equation}\label{eq:toricboundary}
{\bf H}_{z}(u_i, \cdot) =0\qquad {\rm and}\qquad (d{\bf H})_{z}(u_i,u_i) = 2
u_i,
\end{equation}
where the differential $d{\bf H}$ is viewed as a smooth $S^2{\mathfrak
t}^*\otimes {\mathfrak t}$-valued function on $\Delta$\textup;
\item \textup{[positivity]} for any point $z$ in the interior of a face
$F \subseteq \Delta$, ${\bf H}_{z}(\cdot, \cdot)$ is positive definite
when viewed as a smooth function with values in $S^2({\mathfrak
t}/{\mathfrak t}_F)^*$.
\end{bulletlist}
These conditions can be formulated in the following alternative way,
cf.~\cite{Abreu1,Do3}: (i) $U$ is smooth and strictly
convex\footnote{In~\cite{Abreu1} the strict convexity condition on the
interior of the proper faces is realized equivalently as a condition on the
determinant of $\mathrm{Hess}\, U$.} on the interior, $F ^0$, of each face
$F$ of $\Delta$; (ii) if $F = \bigcap _{i} F _i$, where $F _i$ is a
codimension one face of $\Delta$ on which $\ip{u_i,z}+c_i=0$, then, in some
neighbourhood of $F ^0$ in $\Delta$, $U$ is equal to $\frac12\sum_i
(\ip{u_i,z}+c_i)\log(\ip{u_i,z}+c_i)$ up to a smooth function.
\smallskip

We denote by $\cS(\Delta)$ the space of all symplectic potentials on
$\Delta$ defined either way.  For any $U$ in $\cS(\Delta)$, we thus 
get a $\T$-invariant,
$\omega _V$ compatible, K\"ahler metric $g _V$ on $V$.

\smallskip

The second building block of the generalized Calabi construction 
consists in using $g _V$ to
construct a  K\"ahler metric $g_W$ on the variety $W$, 
with respect to which the
restricted $\T$-action is rigid and semisimple. This part of the
construction only appears in the situation ``with blow-down'' and
relies in a crucial way on \cite[Prop.~2]{hfkg2}. Recall that $W$ was 
obtained by a {\it restricted symplectic quotient process}, which
ultimately amounts to  a blow-down of $\hat{W} = P _0 \times _{\T} V$,
where $P _0$ is a $\T$-principal bundle over $\prod _{b \in
  \mathcal{B}}  \mathbb{C} P ^{d _b}$, cf. \S\ref{s:rigid}. The
  construction of $g _W$ then requires the choice of a
connection $1$-form $\theta_0$ on $P_0$, with curvature 
$d\theta_0= \sum_{b \in \cB}
\omega_b \otimes u_b$ where, we recall, $\omega_b$ is the (normalized)
Fubini--Study metric on $\C P^{d_b}$ of scalar curvature $2d_b(d_b+1)$, and
$u_b \in {\mathfrak t}$ is the inward normal to the codimension one face $F_b
\subset \Delta$ (satisfying $\langle u_b, z \rangle + c_b=0$). We
still denote 
by $\theta_0\in\Omega^1({W}^0,\mathfrak t)$ the induced $1$-form on
the open dense subset ${W}^0 :=
P_0 \times_{\T} V^0$ of $\hat{W}$ and we consider the K\"ahler
structure on ${W}^0$ defined by:
\begin{equation}\label{W}
\begin{aligned}
g_W&= \sum_{b \in \cB}(\ip{u_b,z}+c_b)g_b+\ip{dz,{\bf G},dz}+\ip{\theta_0, {\bf H},\theta_0},\\
\omega_{W} &= \sum_{b \in \cB} (\ip{u_b,z}+c_b)\omega_b+\ip{dz\wedge\theta_0},  \qquad\qquad d\theta_0= \sum_{b \in \cB} \omega_b \otimes u_b,
\end{aligned}
\end{equation}
with ${\bf G} = {\rm Hess}(U) = {\bf H}^{-1}$. 
Clearly, the K\"ahler structure $(g_W, \omega_{W})$ is well-defined 
on $W^0 = P_0\times_{\T}
V^0$. As shown in \cite{hfkg2}, the pair $(g_W, \omega_{W})$ smoothly
extends to $\hat{W}$ --- not as a K\"ahler structure however --- and
descends to a smooth, $\T$-invariant,  K\"ahler
structure on $W$.

\smallskip

The third and last building block of the generalized Calabi construction similarly consists in constructing a suitable 
K\"ahler structure on $M ^0 = \hat{P} \times _{\T} V ^0$, via the
choice of a connection $1$-form $\hat \theta$ on $\hat P$, 
with curvature (covered by) $\sum_{a \in \mathcal{A}} \omega _a
\otimes p _a + \sum _{b \in \mathcal{B}} \omega _b \otimes u _b$. Then
the restriction of $(\hat{P}, \hat{\theta})$ to each fibre of $\hat{S}
\to S$ is isomorphic to $(P _0, \theta _0)$ over $\prod _{b \in
  \mathcal{B}}  \mathbb{C} P ^{d _b}$. 
Still denoting by ${\hat
\theta} \in\Omega^1({M}^0,\mathfrak t)$ the induced
$1$-form on $M^0= \hat P \times_{\T} V^0$, we consider  the K\"ahler
structure $(g,\omega)$ on $M^0$ defined by:
\begin{equation}\label{M}
\begin{aligned}
g&= \sum_{j=1}^N (\ip{p_j,z}+c_j)g_j+\ip{dz,{\bf G},dz}+\ip{{\hat \theta}, {\bf H},{\hat \theta}},\\
\omega&= \sum_{j=1}^N (\ip{p_j,z}+c_j)\omega_j+\ip{dz\wedge{\hat \theta}},\qquad\qquad
d{\hat \theta}=\sum_{j=1}^N \omega_j\otimes p_j,
\end{aligned}
\end{equation}
where:
\begin{bulletlist}
\item ${\bf G} = {\rm Hess} (U)= {\bf H}^{-1}$, where $U$ is the
  symplectic potential of the chosen toric K\"ahler structure $g _V$
  on $V$; 
\item for each $b \in \cB$,  $p_b = u_b$ and the real number $c_b$ is such that $\ip{p_b,z}+c_b=0$ on the codimension one face $F_b$;
\item for each $a \in \cA$, $\ip{p_a,z}+c_a$ is
positive on $\Delta$.
\end{bulletlist}
Clearly, \eqref{M} defines a smooth tensor on $\hat M$ and it is shown in
\cite[Thm.~2]{hfkg2} that it is the pullback of a smooth metric on the
blow-down $M$. Indeed, this is obvious in the case when the fibre bundle
$\hat S \to S$ is trivial (for example taking $\hat M$ be simply connected,
as in \cite{hfkg2}). Then, there exists a principal $\T$-bundle over $S$
with connection form $\theta$ and curvature $d\theta = \sum_{a \in \cA}
\omega_a \otimes p_a$ and the restriction of $\hat \theta$ to $S_0=
\prod_{b\in \cB} \C P^{d_b}$ gives rise to a principal $\T$-bundle $P_0$
over $S^0$ with connection $1$-form $\theta_0$ and curvature $d\theta_0 =
\sum_{b \in \cB} \omega_b \otimes u_b $. Thus, $M \cong P \times_{\T} W$,
$M^0 \cong P \times_{\T} W^0$ where $W^0 = P_0 \times_{\T} V^0$.  It
follows that the metric \eqref{M} restricts on each $W^0$ fibre to the
metric $(g_{W}, \omega_{W})$ defined by \eqref{W}; as $(g_{W}, \omega_{W})$
compactifies smoothly on $W$, and $\ip{p_a,z}+c_a$ are strictly positive on
$M$, \eqref{M} defines a K\"ahler structure on $M$. To handle the general
case, one can consider the universal covers of $\hat M$ and $M$ and use the
previous argument, noting that the smooth extension of the metric is a
local property; a direct argument in the case of the projective bundles
described in Sect.~\ref{s:projective bundles} can be given along the lines
of \cite[\S~1.3]{ACGT}. This completes the {\it generalized Calabi
construction} according to \cite{hfkg2}.

\smallskip

Assuming that the metrics $(g_j,\omega_j)$, the
connection $1$-form $\hat \theta$, the polytope $\Delta$
and the constants $(p_j,c_j)$ are all fixed, \eqref{M} defines a family of
K\"ahler metrics parametrized by symplectic potentials $U \in \cS(\Delta)$
(or, equivalently, by toric K\"ahler metrics on $(V,\omega_V,\T)$). We note
that for this family, the symplectic $2$-form $\omega$ remains unchanged, so
we obtain a family of $\T$-invariant $\omega$-compatible K\"ahler metrics
corresponding to {\it different} complex structures. However, any two such
complex structures are biholomorphic, under a $\T$-equivariant diffeomorphism
in the identity component: this is well-known in the case of a symplectic
toric manifolds (i.e., on $(V,\omega_V, \T)$)~see~\cite{Abreu1,Do2}, and the
same argument holds (fibrewise) on $W$ and $M$, see
\cite[\S~1.4]{ACGT}. 
The pullbacks of the
symplectic form $\omega$ under such diffeomorphisms introduce a K\"ahler class
$\Omega$ on a fixed complex manifold $(M,J)$ (we can take $J$ to be the
complex structure on $M$ introduced in Definition~\ref{blow-down}: it
corresponds to the {\it standard} symplectic potential $U_0$,
see~\cite{Abreu1,Guillemin}).  
\begin{defn} K\"ahler structures $(g,\omega)$ arising from the
  generalized Calabi construction on a rigid toric bundle $M$,  depending on the choice of a symplectic
  potential $U$ on the corresponding Delzant polytope
$\cS(\Delta)$, whose explicit expression is given by
  \eqref{M} on $M ^0$, 
are called {\it compatible}. The corresponding K\"ahler classes are
accordingly called {\it compatible K\"ahler classes}.
\end{defn}

We shall further assume that the metrics $(g_j,\omega_j)$ are fixed and have
constant scalar curvature $Scal_j$ (with $Scal_b = 2d_b(d_b+1)$ for $b \in
\cB$),\footnote{Presumably, the K\"ahler metrics $(g_j,\omega_j)$ must be CSC
in order to obtain an extremal K\"ahler metric $(g,\omega)$ as above. We do
not prove this here, but this fact has been established for $\ell=1$ in
\cite[Prop.~14]{hfkg-survey}.} and that $\Delta$ and $p_j$ are fixed. Recall
that for $b\in \cB$, the constants $c_b$ are also fixed by requiring 
$\langle u_b, z \rangle + c_b=0$ on the codimension one face $F_b \subset
\Delta$. The real constants $c_a, a \in \cA$ can vary (on a given manifold
$(M,J)$) and they parametrize the compatible K\"ahler classes.

\subsection{The isometry Lie algebra}\label{isometry}

For a compact K\"ahler manifold $(M,g)$, we denote by $\mathfrak{i}_0(M,g)$
the Lie algebra of all Killing vector fields with zeros; 
this is equivalently the Lie algebra of all hamiltonian Killing vector fields.

The following result has been established in the case $\ell=1$ in
\cite[Prop.~3]{ACGT} and its proof generalizes to the general case. For the
convenience of the Reader, we reproduce the argument from \cite{ACGT}.
\begin{lemma}\label{isometry algebra}
Let $(g,\omega)$ be a compatible K\"ahler metric on $M$, where the stable
complex quotient $\hat S$ is equipped with the local product K\"ahler
metric $(g_{\hat S},\omega_{\hat S})$ covered by
$\prod_{j=1}^N(S_j,\omega_j)$.  Denote by ${\hat p}\colon M^0 \to \hat S$
the principal $\T^c$-fibre structure of the regular part $M^0$ of $\T$
action on $M$. Let $\mathfrak z(\T,g)$ be the centralizer in
$\mathfrak{i}_0(M,g)$ of the $\ell$-torus $\T$.

Then the vector space $\mathfrak z(\T,g)$ is the direct sum of a lift of
$\mathfrak{i}_0(\hat S,g_{\hat S})$ and the Lie algebra $\mathfrak{t}
\subset \mathfrak{i}_0(M,g)$ of $\T$ in such a way that the natural
homomorphism ${\hat p}_*\colon \mathfrak z(\T,g)\to \mathfrak{i}_0(S, g_S)$ is
a surjection.
\end{lemma}
\begin{proof} Denote by
$K={\rm grad}_{\omega} z \in {C}^{\infty}(M,TM) \otimes \mathfrak{t}^*$ the
family of hamiltonian Killing vector fields generated by $\T$: thus, the span
of $K$ realizes the Lie algebra $\mathfrak t$ of $\T$ as a subalgebra of
$\mathfrak{i}_0(M,g)$.

Let $X$ be a holomorphic vector field on $\hat S$ which is hamiltonian with
respect to $\omega_{\hat S}$; then the projection $X_j$ of $X$ onto the
distribution ${\mathcal H}_j$ (induced by $TS_j$ on the universal cover
$\prod_{j=1}^N S_j$ of $S$) is a Killing vector field with zeros, so
$\iota^{\vphantom{x}}_{X_j}\omega_{\hat S}=-df_j$ for some function $f_j$
(with integral zero). Thus $\sum_{j=1}^N f_jp_j$ is a family of hamiltonians
for $X$ with respect to the family of symplectic forms covered by
$\sum_{j=1}^N \omega_j\otimes p_j$: since this is the curvature $d\hat \theta$
of the connection on $M^0$, $X$ lifts to a holomorphic vector field $\tilde X
=X_H+\sum_{j=1}^N f_j\ip{p_j, K}$ on $M^0$, which is hamiltonian with
potential $\sum_{j=1}^N(\ip{p_j, z}+c_j)f_j$ and commutes with the components
of $K$. (Here $X_H$ is the horizontal lift to $M^0$ with respect to $\hat
\theta$.)  As the metric $g$ extends to $M$ and $\tilde X$ is Killing with
respect to $g$, it extends to $M$ too (note that $M\setminus M^0$ has
codimension $\ge 2$). It is not difficult to see that $\tilde X$ has zeros on
$M$ (in fact, if $s_0\in \hat S$ is a zero of $X$ then $\tilde X -
\sum_{j=1}^N f_j(s_0)\langle p_j, K\rangle$ vanishes on $M^0$) so that $\tilde
X$ is an element of $\mathfrak{i}_0(M,g)$. Of course, this shows that the
Killing potential $\sum_{j=1}^N(\ip{p_j, z}+c_j)f_j$ extends as a smooth
function on $M$.

Conversely, any $\tilde X \in \mathfrak z(\T,g)$ is a $\T^c$-invariant
holomorphic vector field, so its restriction to $M^0$ is projectable to a
holomorphic vector field $X\in \mathfrak{h}_0(\hat S)$. This allows to reverse
the above arguments: for $\tilde X=X_H+f\ip{p, K} + h J \ip{q,K}$ (where $p,q
\in \mathfrak t$ and $f, g \in \cC^{\infty}(\hat S)$) be Killing with respect
to the metric \eqref{M}, we must have $q=0$ and $X$be Killing with respect to
$g_{\hat S}$. Such a vector field maps to zero iff it comes from a constant
multiple of $K$. This gives a projection to $\mathfrak{i}_0(\hat S,g_{\hat
S})$ splitting the inclusion just defined.
\end{proof}
This is the main ingredient in the proof of the following result.

\begin{prop}\label{max-torus} Let $(J, g,\omega)$ be a compatible K\"ahler
metric on $M$ where the stable quotient $\hat S$ is endowed with a local
product K\"ahler structure $(g_{\hat S}, \omega_{\hat S})$, covered by
$\prod_{j=1}^N(S_j,\omega_j)$ with $(S_j, \omega_j)$ having constant scalar
curvature.

Then $g$ is invariant under a maximal torus $G$ of the reduced automorphism
group ${\widetilde {\rm Aut}}_0(M,J)$.
\end{prop}
\begin{proof} Let $G$ be a maximal torus in the group of hamiltonian
isometries ${\rm Isom}_0(M,g)$, containing the $\ell$-torus $\T$. By
Lemma~\ref{isometry algebra}, $G$ is the product of a maximal torus in the
group of hamiltonian isometries ${\rm Isom}_0(\hat S, g_{\hat S})$ and the
$\ell$-torus $\T$. Denote by $\mathfrak g \subset \mathfrak{i}_0(M,g)$ the
corresponding Lie algebra.  We are going to show that ${\mathfrak g}^{\C} =
\mathfrak g + J \mathfrak g$ is a maximal abelian subalgebra of ${\mathfrak
h}_0(M,J)$.

As in the proof of Lemma~\ref{isometry algebra}, we consider natural
homomorphism ${\hat p}_* \colon \mathfrak{z}(\T,J) \mapsto \mathfrak{h}_0(\hat
S)$ from the centralizer $\mathfrak{z}(\T,J)$ of $\T$ in $\mathfrak{h}_0(M,J)$
to $\mathfrak{h}_0(\hat S)$. The proof of Lemma~\ref{isometry algebra} shows
that the restriction of ${\hat p}_*$ to $\mathfrak z(\T,g)$ is surjective onto
$\mathfrak{i}_0(\hat S, g_{\hat S})$.

By assumption, the induced K\"ahler metric $(g_{\hat S}, \omega_{\hat S})$ on
$\hat S$ is of constant scalar curvature, so by the Lichnerowicz--Matsushima
theorem~\cite{lichne,matsushima}, $\mathfrak{h}_0(\hat S)$ is the
complexification of $\mathfrak{i}_0(\hat S, g_{\hat S})$. It follows that
${\hat p}_* \colon \mathfrak{z}(\T,J) \to {\mathfrak h}_0(S)$ is also
surjective. As $\mathfrak{g} \subset \mathfrak z(\T,g)$ is a maximal abelian
subalgebra, its projection to $\mathfrak{i}_0(S, g_{S})$ must also be a
maximal abelian subalgebra, so is then the image ${\hat
p}_*(\mathfrak{g}^{\C}) \subset \mathfrak{h}_0(\hat S)$ (by using the
Lichnerowicz--Matsushima theorem again).  It follows that ${\mathfrak
g}^{\C}\subset \mathfrak{h}_0(M,J)$ is maximal abelian iff $\mathfrak{g}^{\C}
\cap \mathfrak{h}_{\hat S}(\hat M)$ is a maximal abelian subalgebra of the
complex algebra of fibre-preserving holomorphic vector fields
$\mathfrak{h}_{\hat S}( \hat M)$. But the fibre $V$ is a toric variety under
$\T$, so $\mathfrak{g}^{\C} \cap \mathfrak{h}_{\hat S}(\hat M)=
\mathfrak{t}^{\C}=\mathfrak{t} + J \mathfrak{t}$, which is clearly a maximal
abelian subalgebra of $\mathfrak{h}(V,J_V)$ and hence also of
$\mathfrak{h}_{\hat S}(\hat M)$. \end{proof}

\subsection{The extremal vector field}\label{extremal vector field}

For convenience, we will introduce at places a basis of $\mathfrak t$
(resp. of ${\mathfrak t}^*$), for example by taking $\ell$ generators of the
lattice $\Lambda$ (where $\T = \mathfrak t / 2\pi \Lambda$). This identifies
the vector space $\mathfrak t$ with $\R^{\ell}$ (and ${\mathfrak t}^*$ with
$(\R^{\ell})^*$), and fixes a basis of Poisson commuting hamiltonian Killing
fields $K_1, \ldots , K_{\ell}$ in $K$.  Thus, a $S^2{\mathfrak t}^*$-valued
function ${\bf H}$ on $\Delta$ can be seen as an $\ell \times \ell$-matrix of
functions $(H_{rs})= {\bf H}$ on $\Delta$. Similarly, we write $z=(z_1,
\ldots, z_{\ell})$ for the momentum coordinates with respect to $K_1, \ldots,
K_{\ell}$.

An important technical feature of the K\"ahler metrics given by the
generalized Calabi construction \eqref{M} is the simple expression of their
scalar curvature in terms of the geometry of $(V,g_V)$ and $(\hat S, g_{\hat
S})$ (see e.g.~\cite[p. 380]{hfkg1}):
\begin{equation}\label{scal}
\Scal_g = \sum_{j =1}^N\frac{\Scal_j}{\ip{p_j,z}+c_j}-\frac 1{p(z)}\sum_{r,s=1}^{\ell}
\frac{\del^2}{\del z_r\del z_s} (p(z) H_{rs}),
\end{equation}
where $p(z)=\prod_{j=1}^N (\ip{p_j,z}+c_j)^{d_j}$.  This formula generalizes
the expression obtained by Abreu~\cite{Abreu0} in the toric case (when $\hat
S$ is a point).

Another immediate observation is that the volume form ${\rm
Vol}_{\omega}=\omega^m$ is given by\begin{equation}\label{Vol} \omega^m
=p(z)\Bigl(\omega_{\hat S}^d \wedge \ip{dz\wedge\hat \theta}^{\wedge\ell}
\Bigr) = p(z)\Bigl(\bigwedge_{j} \omega_j^{\wedge d_j}\Bigr) \wedge
\ip{dz\wedge\hat \theta}^{\wedge\ell},
\end{equation}
where $\sum_{j=1}^N d_j=d=m-\ell$.  It follows that integrals over $M$ of
functions of $z$ (pullbacks from $\Delta$) are given by integrals on $\Delta$
with respect to the volume form $p(z)\,dv$, where $dv$ is the (constant)
euclidean volume form on $\mathfrak t^*$, obtained by wedging any generators
of the lattice $\Lambda$.

\smallskip

We now recall the definition in \cite{FM} of the {\it extremal vector
field} of a compact K\"ahler manifold $(M,J,g,\omega)$. Let $G$ be a
maximal connected compact subgroup of the reduced group of automorphisms
${\widetilde {\rm Aut}}_0(M,J)$.\footnote{By a well-known result of
Calabi~\cite{cal-two}, any extremal K\"ahler metric must be invariant under
such a $G$.} Following \cite{FM}, the {\it extremal vector field} of a
$G$-invariant K\"ahler metric $(g,J,\omega)$ on $M$ is the Killing vector
field whose Killing potential is the $L^2$-projection of the scalar
curvature $Scal_g$ of $g$ to the space $\mathfrak g_{\omega}$ of all
Killing potentials (with respect to $g$) of elements of the Lie algebra
${\mathfrak g}$. Futaki and Mabuchi~\cite{FM} showed that this definition
is independent of the choice of a $G$-invariant K\"ahler metric within the
given K\"ahler class $\Omega=[\omega]$ on $(M,J)$. Since the extremal
vector field is necessarily in the centre of ${\mathfrak g}$, it can be
equally defined if we take $G$ be only a maximal torus in ${\widetilde {\rm
Aut}}_0(M,J)$. This remark is relevant to the K\"ahler metrics \eqref{M} as
we have already shown in Proposition~\ref{max-torus} that they are
automatically invariant under such a torus $G$. In this case, by
Lemma~\ref{isometry algebra}, $\mathfrak{g}_{\omega}$ is the direct sum of
${\mathfrak t}_{\omega}$ (which in turn is identified to the space of
affine functions of $z$) and a subspace of Killing potentials of zero
integral of lifts of Killing vector fields on $(\hat S, g_{\hat S})$. We
have shown in the proof of Lemma~\ref{isometry algebra} that the later
potentials are all of the form $\sum_j(\ip{p_j, z}+c_j)f_j$ where $f_j$ is
a function on $\hat S$ of zero integral with respect to $\omega_{\hat
S}^{d}$. As the scalar curvature of a compatible metric is a function of
$z$ only (see \eqref{scal}, we assume $Scal_j$ are constant) it follows
from \eqref{Vol} that the $L^2$-projection of $Scal_g$ to ${\mathfrak
g}_{\omega}$ lies in $\mathfrak{t}_{\omega}$. This shows that the extremal
vector field lies in $\mathfrak t$ and that the projection of $\Scal_g$
orthogonal to the Killing potentials of $g$ takes the form:
\begin{equation*}
\Scal_g^{\perp} = \ip{A,z}+B+\Scal_g,
\end{equation*}
where
\begin{gather}\label{extremal-field}
\begin{cases}
\sum_s \alpha_s A_s + \alpha B +2\beta &=0,\\
\sum_s \alpha_{rs} A_s + \alpha_r B +2\beta_r &=0,
\end{cases}\\
\tag*{with} \nonumber
\alpha=\int_\Delta p(z) dv, \qquad \alpha_r =\int_\Delta z_r p(z) dv,
\qquad \alpha_{rs} =\int_\Delta z_r z_s p(z) dv,\\ \nonumber
\begin{split} \nonumber
\beta&=\frac12\int_{\Delta} \Scal_g p(z)dv
=\int_{\del\Delta} p(z) d\sigma +\frac12
\int_\Delta \Bigl(\sum_j \frac{\Scal_j}{\ip{p_j,z}+c_j}\Bigr) p(z) dv,\\ \nonumber
\beta_r&=\frac12\int_{\Delta} \Scal_g z_r p(z)dv
=\int_{\del\Delta} z_r p(z) d\sigma +\frac12
\int_\Delta \Bigl(\sum_j \frac{\Scal_j}{\ip{p_j,z}+c_j}\Bigr) z_r p(z) dv. \nonumber
\end{split}
\end{gather}
Here $d\sigma$ is the $(\ell-1)$-form on $\del\Delta$ with $u_i\wedge
d\sigma=-dv$ on the face $F_i$ with normal $u_i$. These formulae are immediate
once one applies the divergence theorem and the boundary conditions
\eqref{eq:toricboundary} for ${\bf H}$, noting that the normals are inward
normals, which introduces a sign compared to the usual formulation of the
divergence theorem.

The extremal vector field of $(M,g,J,\omega)$ is $-\ip{A,K}$, where $K\in
C^\infty(M,TM)\otimes\mathfrak t^*$ is the generator of the $\mathbb T$
action.

\subsection{The extremal equation and stability of its solutions under small perturbation}

It follows from the considerations in Sect.~\ref{extremal vector field} that
on a given manifold $M$ of the type we consider, finding a {\it compatible}
extremal K\"ahler metric $(g,\omega)$ of the form \eqref{M} reduces to solving
the equation (for a unknown symplectic potential $U \in {\cS}(\Delta)$)
\begin{equation}\label{operator}
\langle A, z\rangle + B + \sum_{j =1}^N \frac{Scal_j}{c_j +
\langle p_j,z \rangle} - \frac{1}{p(z)} \sum_{r,s} \frac{\partial^2}{\partial
z_r \partial z_s} \big(p(z)H_{rs}\big) =0,
\end{equation}
where 
\begin{bulletlist}
\item $(H_{rs})={\bf H}= ({\rm Hess} (U))^{-1}$;
\item $(c_j, p_j, Scal_j)$ are fixed constants;
\item $p(z) = \prod_{j=1}^N (c_j + \langle p_j,z \rangle)^{d_j}$ is strictly
positive on $\Delta^0$ but vanishes on the blow-down faces $F_b, \, b \in
\cB$;
\item $A$ and $B$ are expressed in terms of $(c_j, p_j, Scal_j)$ by
\eqref{extremal-field}.
\end{bulletlist}

\smallskip
Recall from Sect.~\ref{s:generalised-calabi} that the real constants $c_a, \,
a \in \cA$ parametrize compatible K\"ahler classes on a given manifold $M$.  A
general result of LeBrun--Simanca~\cite{Le-Sim1} affirms that K\"ahler classes
admitting extremal K\"ahler metric form an open subset of the K\"ahler cone.
We want to obtain a relative version of this result, by showing that {\it
compatible} K\"ahler classes which admit a {\it compatible} extremal K\"ahler
metric is an open condition on the parameters $c_a$.

We will state and prove our stability result in a slightly more general
setting, by considering \eqref{operator} as a family of differential operators
on $\cS(\Delta)$, parametrized by $\lambda \in \{ (c_a,p_a, Scal_a), \, a \in
\cA \}$ (thus $\lambda$ takes values in a $(2+\ell) |\cA|$-dimensional
euclidean vector space). For any $\lambda$ such that $\langle p_a, z \rangle +
c_a >0$ on $\Delta$, we consider
\begin{equation}\label{P-lambda}
P_{\lambda} (U) =\langle A_{\lambda}, z\rangle + B_{\lambda} + \sum_{j =1}^N \frac{Scal_j}{c_j +
\langle p_j,z \rangle} - \frac{1}{p_{\lambda}p_0} \sum_{r,s=1}^{\ell} \frac{\partial^2}{\partial
z_r \partial z_s} \big(p_{\lambda}p_0 H_{rs}\big), 
\end{equation}
where $(H_{rs})= {\rm Hess}(U)^{-1}$, $p_{\lambda}(z) = \prod_{a \in \cA}
(\langle p_a, z \rangle + c_a)^{d_a}$, $p_0(z) = \prod_{b \in \cB} (\langle
u_b, z \rangle + c_b)^{d_b}$, and $A_{\lambda}, B_{\lambda}$ are introduced by
\eqref{extremal-field}. The central result of this section is the following
one.

\begin{prop}\label{perturbation} Let $(g_0, \omega_0)$ be a compatible
extremal K\"ahler on $M$, with symplectic potential $U_0$ and parameters
$\lambda_0=(c_a^0,p_a^0,Scal_a^0), \, a \in \cA$. Then there exists
$\varepsilon >0$ such that for any $\lambda$ with $|\lambda -
\lambda_0|<\varepsilon$ there exists a symplectic potential $U_{\lambda} \in
\cS(\Delta)$ such that $P_{\lambda} (U_{\lambda})=0$ on $\Delta^0$.
\end{prop}

The proof of this proposition has several steps and will occupy the rest of
this section.

\smallskip

It is not immediately clear from \eqref{P-lambda} that $P_{\lambda}$ is a
well-defined differential operator: in the presence of blow-downs, the terms
$\frac{Scal_b}{c_b + \langle p_b,z \rangle}$ and $\frac{1}{p_{0}(z)}$ become
degenerate on the boundary of $\Delta$.\footnote{This does not affect the
principal part of $P_{\lambda}$, which is concentrated in the scalar curvature
$Scal_V = -\sum_{r,s} \frac{\partial^2}{\partial z_r \partial z_s} H_{rs}$ of
the induced K\"ahler metric $g_{V}$ on $V$~\cite{Abreu0}, and is manifestly
independent of $\lambda$.}  Of course, for $\lambda=\lambda_0$ we know from
\eqref{operator} that $P_{\lambda_0}(U)= Scal_g ^{\perp}$ where $g$ is the
compatible metric on $M$ corresponding to $U$, and $Scal_g ^{\perp}$ is the
$L^2$-projection of the scalar curvature to the space of functions orthogonal
to the Killing potentials of $g$. However, for generic values of $\lambda$ the
data $(c_a, p_a, Scal_a)$ are not longer associated with a compatible K\"ahler
class on a smooth manifold: for this to be true $p_a$ and $Scal_a$ must
satisfy {\it integrality} conditions. To overcome this technical difficulty,
we are going to rewrite our equation on the smooth compact manifold $W$. (Note
that for $b \in \cB$, $p_b=u_b, c_b, Scal_b= 2d_b(d_b+1)$ are fixed in our
construction.)

Recall from Sect.~\ref{s:generalised-calabi} that any symplectic potential $U
\in \cS(\Delta)$ introduces a compatible K\"ahler metric $(g_W,\omega_W)$ on
the manifold $W$ obtained by blowing down $\hat W= P_0 \times_{\T} V$. Thus,
$(W, g_{W}, \omega_{W})$ itself is obtained by the generalized Calabi
construction with $S$ being a point.

By a well-known result of G. W. Schwarz~\cite{Schwarz}, the space
${C^{\infty}(V)}^{\T}$ of $\T$-invariant smooth functions on the toric
symplectic manifold $(V, \omega_V, \T)$ is identified with the space of
pullbacks (via the momentum map $z$) of smooth functions $C^{\infty}(\Delta)$
on $\Delta$; similarly, the space of smooth $\T$-invariant functions on $W$
(resp. on $M$) which are constant on the inverse images of the momentum map
$z$ is identified with the space $C^{\infty}(\Delta)$. We will use implicitly
these identification throughout. Occasionally, when we want to emphasize the
dependence of this identification on $z$, we will denote these isomorphisms by
$S_z$. With this convention, we have

\begin{lemma} \label{geometric} Let $U\in \cS(\Delta)$ be a symplectic
potential of a compatible K\"ahler metric $g_V$ on $(V, \omega_V, \T)$ and
$(g_W,\omega_W)$ be the corresponding compatible K\"ahler metric on $W$.
Then, for any $\lambda$ such that $\langle p_a, z \rangle + c_a >0$ on
$\Delta$,
\begin{equation*}
\begin{split}
P_{\lambda} (U)= & \langle A_{\lambda}, z\rangle + B_{\lambda} +
\sum_{a \in \cA} \frac{Scal_a}{c_a + \langle p_a,z \rangle} + Scal_{W} \\
& - \frac{1}{p_{\lambda}(z)} \sum_{r,s=1}^{\ell} \Big(\Big(\frac{\partial^2
p_{\lambda}}{\partial z_r \partial z_s}\Big)(z) g_W (K_r,K_s)\Big) \\
& +  \frac{2}{p_{\lambda}(z)} \sum_{r=1}^{\ell} \Big(\Big(\frac{\partial p_{\lambda}}{\partial
z_r}\Big)(z)\Delta_{W}z_r\Big),
\end{split}
\end{equation*}
where $Scal_W$ and $\Delta_W$ respectively denote the scalar curvature and the
riemannian laplacian of $g_W$, and $dz_r = - \omega_W(K_r, \cdot)$.
\end{lemma}
\begin{proof} We work on the open dense subset $W^0=P_0\times_{\T} V^0$ where
the compatible metric $(g_{W}, \omega_{W})$ takes the explicit form \eqref{W}.
The formula \eqref{scal} for the scalar curvature of the compatible metric
$g_{W}$ then specifies to
\begin{equation*}
Scal_{W} = \sum_{b \in \cB}\frac{Scal_b}{\ip{p_b,z}+c_b}-\frac
1{p_0(z)}\sum_{r,s=1}^{\ell} \frac{\del^2}{\del z_r\del z_s} (p_0(z) H_{rs}).
\end{equation*}
Still using the explicit form \eqref{W} of the K\"ahler structure, we
calculate that for the pullback to $W$ of a smooth function $f(z)$ on $\Delta$
\begin{equation}\label{eq:ddc}
\begin{split}
dd_W^c f & = d \Big( \sum_{r,s=1}^{\ell} \frac{\partial f}{\partial z_s}
H_{rs} (\theta_0)_r \Big) \\ & = \sum_{k,r,s=1}^{\ell} \frac{\partial
}{\partial z_k}\Big(\frac{\partial f}{\partial z_s} H_{rs}\Big) dz_k \wedge
(\theta_0)_r + \sum_{b\in \cB} \Big(\sum_{r,s=1}^{\ell}\frac{\partial
f}{\partial z_s}H_{rs}p_{br}\Big)\omega_b,
\end{split}
\end{equation}
where the decompositions $\theta_0=((\theta_{0})_1, \ldots,
(\theta_{0})_{\ell})$ and $p_b=(p_{b1}, \ldots, p_{b\ell})$ are with respect
to the chosen basis of $\mathfrak t$ and ${\mathfrak t}^*$.  Wedging with
$\omega_W$, we obtain the following expression for the laplacian
\begin{equation}\label{laplacian}
\Delta_W f = - \frac{1}{p_0(z)} \sum_{r,s=1}^{\ell} \frac{\partial}{ \partial
z_r} \Big(p_{0}(z) \frac{\partial f}{\partial z_s}H_{rs}\Big).
\end{equation}
Specifying \eqref{laplacian} to $f=z_r$ and putting the above formulae back in
\eqref{P-lambda} implies the lemma. \end{proof}

Note that $\frac{1}{p_{\lambda}(z)}$ and $\frac{Scal_a}{c_a + \langle p_a,z
\rangle}$ pull back to smooth functions on $W$ for $\lambda$ such that $c_a +
\langle p_a,z \rangle >0$ on $\Delta$, and $A_{\lambda}$ and $B_{\lambda}$ are
well-defined and depend smoothly on $\lambda$ (at least for $\lambda$ close to
$\lambda_0$). Thus, Lemma~\ref{geometric} implies that $P_{\lambda}$ is a
fully non-linear $4$-th order differential operator which depends smoothly on
$\lambda$ (for $\lambda$ sufficiently close to $\lambda_0$). It follows that
$P_\lambda (U) \in C^{\infty}(\Delta)$ for any $U \in \cS(\Delta)$.

\smallskip

Our problem is formulated in terms of compatible K\"ahler metrics on $V$ (or,
equivalently, on $W$ and $M$) with respect to a fixed symplectic form
$\omega_V$ (resp. $\omega_{W}$ and $\omega$). This introduces the space of
symplectic potentials $\cS(\Delta)$ where we have to work with smooth
functions on $\Delta^0$ which have a prescribed boundary behaviour on
$\partial \Delta$. Our lack of understanding of the convergence in this space
(with respect to suitable Sobolev norms) leads us to make an additional
technical step and reformulate our initial problem as an existence result on a
suitable subspace of the space ${\cM_{\Omega}(M)}^G \cong \{ f \in {C^{\infty}_0(M)}^G : \omega_0 + dd^c f >0 \}$ of $G$-invariant K\"ahler
metrics in the K\"ahler class of $(g_0, J_0, \omega_0)$, where ${C^{\infty}_0(M)}^G$ denotes the space of $G$-invariant smooth functions on $M$
of zero integral with respect to $\omega_0^m$ (thus ${\cM_{\Omega}(M)}^G$ is
viewed as an open set in ${C^{\infty}_0(M)}^G$ with respect to $||\cdot
||_{{C}^2}$). Once this interpretation is achieved, we will apply the implicit
function theorem along the lines of the proof of Lemma~\ref{stability}.

First of all, note that the Frech\'et space $C^{\infty}(\Delta)$ pulls back
via $z$ to a closed subspace in ${C^{\infty}(V)}^{\T}$, ${C^{\infty}(W)}^T$ and
${C^{\infty}(M)}^G$, where $T$ (resp. $G$) is a maximal torus in ${\widetilde
{\rm Aut}}_0(W)$ (resp. ${\widetilde {\rm Aut}}_0(M)$) containing $\T$, as in
Proposition~\ref{max-torus}: this follows easily from the description of the
Lie algebras of $T$ and $G$ given in Lemma~\ref{isometry
algebra}. Furthermore, by \eqref{Vol}, the corresponding {\it normalized}
subspaces of functions with zero integral for the measures
$p_{\lambda}(z)p_0(z){\rm Vol}_{\omega_V^0}$, $p_{\lambda}(z){\rm
Vol}_{\omega_W^0}$ and ${\rm Vol}_{\omega_0}$, respectively, are identified
with the space ${C}_0^{\infty}(\Delta)$ of smooth functions of zero integral
with respect to the volume form $d\mu_0= p_{\lambda_0}(z)p_0(z) dv$ on
$\Delta^0$: this normalization will be used throughout.

Secondly, to adopt the classical point of view of K\"ahler metrics within a
given K\"ahler class on a fixed complex manifold, we consider the Fr\'echet
space ${\cM_{\Omega}(V)}^{\T} \cong \{ f \in {C}_0^{\infty}(\Delta) :
\omega_{V}^0 + dd^c_V f >0 \}$ of $\T$-invariant K\"ahler metrics in the
K\"ahler class $\Omega=[\omega_V^0]$, where the complex structure on $V$
(resp. on $W$ and $M$) is determined (and will be fixed throughout) by the
initial compatible metric $(g_0, \omega_0)$; similarly, we introduce the
spaces ${\cM_{\Omega}(W)}^T$ and ${\cM_{\Omega}(M)}^G$ of K\"ahler metrics in the
given K\"ahler class which are invariant under a maximal torus (see
Proposition~\ref{max-torus}). These three spaces are interrelated by the
generalized Calabi construction as follows.

\begin{lemma}\label{lem:compatible} Let
${\tilde \omega}_V=\omega_{V}^0 + dd^c_V f$ be a K\"ahler metric in
${\cM_{\Omega}(V)}^{\T}$. Then ${\tilde \omega}_{W}= \omega^0_W + dd^c_W f$ and
${\tilde \omega} = \omega_0 + dd^c_M f$ define K\"ahler metrics in
${\cM_{\Omega}(W)}^T$ and ${\cM_{\Omega}(M)}^G$ respectively, such that ${\tilde
\omega}_V, {\tilde \omega}_{W}$ and ${\tilde \omega}$ are linked by the
generalized Calabi construction on $M$, with respect to the data $(\Delta,
\hat S, \hat \theta, \omega_j)$ of the initial metric $\omega_0$, but with
momentum co-ordinate $\tilde z = z +d^c_V f (K)$.
\end{lemma}
\begin{proof} A direct calculation based on the expressions of
$dd^c_V f, dd^c_W f$ and $dd^c_M f$, see \eqref{eq:ddc}; we leave the details
to the reader.
\end{proof}

Lemma~\ref{lem:compatible} allows us to introduce subspaces of {\it
compatible} K\"ahler metrics $\cM^{\rm
comp}_{\Omega}(W)={\cM_{\Omega}(W)}^T\cap {C}_0^{\infty}(\Delta)$ and
$\cM^{\rm comp}_{\Omega}(M)={\cM_{\Omega}(M)}^G\cap {C}_0^{\infty}(\Delta)$
(within a fixed K\"ahler class $\Omega$) and identify each of them with the
space ${\cM_{\Omega}(V)}^{\T}$.  The correspondence which associates to any
${\tilde \omega}_W \in \cM^{\rm comp}_{\Omega}(W)$ (resp. ${\tilde \omega}
\in \cM^{\rm comp}_{\Omega}(M)$) the corresponding symplectic potential
$\tilde U \in \cS(\Delta)$\footnote{For a metric ${\tilde \omega}_V=
\omega_V^0 + dd^c_V f \in \cM^{\T}(V)$ the corresponding symplectic
potential $\tilde U$ is linked to $f$ by a Legendre
transform~\cite{Abreu1,Guillemin}; this is true fibrewise for metrics in
$\cM^{\rm comp}_{\Omega}(W)$ and $\cM^{\rm comp}_{\Omega}(M)$.} allows us
to reformulate our existence problem on the space $\cM^{\rm
comp}_{\Omega}(W)$ as follows: for any $\lambda$ sufficiently close to
$\lambda_0$ (so that $A_{\lambda}$, $B_{\lambda}$ are well-defined and
$\langle p_a, z \rangle + c_a >0$ on $\Delta$), we consider the family of
differential operators on ${\cM_{\Omega}(W)}^T$
\begin{equation}\label{Q-operator}
\begin{split}
Q_{\lambda}({\tilde \omega}_W) & =  \frac{p_{\lambda}(\tilde z)}{p_{\lambda_0}(\tilde z)}\Big[\langle A_{\lambda}, \tilde z \rangle + B_{\lambda} +
\sum_{j=1}^N \frac{Scal_j}{c_j + \langle p_j, \tilde z \rangle} \\
&\quad + {\widetilde {Scal}}_W - \frac{1}{p_{\lambda}(\tilde z)} \sum_{r,s} \Big(\Big(\frac{\partial^2
p_{\lambda}}{\partial z_r \partial z_s}\Big)(\tilde z) {\tilde g}_W(K_r,K_s)\Big) \\
&\quad + \frac{2}{p_{\lambda}(\tilde z)} \sum_{r} \Big(\Big(\frac{\partial p_{\lambda}}{\partial
z_r}\Big)(\tilde z){\widetilde \Delta}_W ({\tilde z}_r)\Big)\Big],
\end{split}
\end{equation}
where $\tilde z= z + d^c_W f(K)$ is the momentum map of $\T$ with respect to
the K\"ahler form ${\tilde \omega}_W= {\omega}_W^0 + dd^c_W f$ of the K\"ahler
metric ${\tilde g}_W$, and ${\widetilde {Scal}}_W$ (resp. ${\widetilde
{\Delta}}_W$) denote the scalar curvature (resp. laplacian) of $\tilde g_W$.
Thus, by Lemmas~\ref{geometric} and \ref{lem:compatible}, {\it any K\"ahler
metric ${\tilde \omega}_{W} \in \cM^{\rm comp}_{\Omega}(W)$ for which
$Q_{\lambda}(\tilde \omega_W)=0$ gives rise to a symplectic potential $\tilde
U \in \cS(\Delta) $ solving $P_{\lambda}(\tilde U)=0$.}

The positive factor $\frac{p_{\lambda}(\tilde z)}{p_{\lambda_0}(\tilde z)}$ in
front of $Q_{\lambda}$ is introduced so that {\it for any compatible metric
$\tilde \omega_W \in \cM^{\rm comp}_{\Omega}(W)$, the function $S_{\tilde z}
(Q_{\lambda}(\tilde \omega_{W}))$ is $L^2$-orthogonal with respect to the
measure $d\mu_0= p_{\lambda_0}p_0dv$ on $\Delta$ to the space of affine
functions on ${\mathfrak t}^*$,} where, we recall, $S_{\tilde z}$ denotes the
identification of $\T$-invariant smooth functions on $W$ which are constant on
the inverse images of $\tilde z$ (equivalently of $z$) with pullbacks via
$\tilde z$ of smooth functions on $\Delta$. Indeed, by Lemma~\ref{geometric},
$p_{\lambda_0}(\tilde z) p_0(\tilde z)Q_{\lambda}(\tilde \omega_{W}) =
P_{\lambda}(\tilde U)p_{\lambda}(\tilde z) p_{0}(\tilde z)$, so integrating by
parts the r.h.s. of \eqref{P-lambda} and using \eqref{eq:toricboundary} we get
\begin{equation*}
\begin{split}
\int_{\Delta} P_{\lambda}(U) f(z) p_{\lambda}(z) p_{0}(z) dv
&= - \int_{\Delta} \langle {\bf H},  {\rm Hess}(f)\rangle p_{\lambda}(z) p_0(z) dv \\ 
&\qquad + \int_{\Delta}\Big(\langle A, z\rangle + B + \sum_{j =1}^N \frac{Scal_j}{c_j + \langle p_j,z \rangle} \Big) f(z)p_{\lambda}(z)p_0(z) dv \\
&\qquad+2 \int_{\partial \Delta} f(z) p_{\lambda}(z)p_0(z) d\sigma ,
\end{split}
\end{equation*} 
which holds for any smooth function $f(z)$. When $f$ is affine, the first term
in the r.h.s is clearly zero, while by the definition \eqref{extremal-field}
of $A_{\lambda}$ and $B_{\lambda}$ the sum of the two other terms is zero too;
our claim then follows by Lemma~\ref{lem:compatible} and the expression
\eqref{Vol} for the volume form of the compatible metric $\tilde \omega_W$.

%This follows easily for $\lambda= \lambda_0$ because $Q_{\lambda_0}(\tilde \omega_W)= Scal^{\perp}_{\tilde g}$ where  $(\tilde g, \tilde \omega)$ is the corresponding K\"ahler metric in  $\cM_\Omega^{\rm comp}(M)$,  and $Scal^{\perp}_{\tilde g}$ denotes the $L^2$-orthogonal projection with respect to $\tilde \omega^m$ of its scalar curvature to the space of functions orthogonal to the Killing potentials of $\tilde g$ (see Sect.~\ref{extremal-field}).

\smallskip
Let $\Pi_{0}$ denote the orthogonal $L^2$-projection of
$C^{\infty}(\Delta)$ to the finite dimensional subspace of affine functions
of ${\mathfrak t}^*$ with respect to the measure $d\mu_0=
p_{\lambda_0}p_0dv$ on $\Delta$, and ${C}_{\perp}^{\infty}(\Delta)$ be the
kernel of $\Pi_0$. We then consider the map $\Psi\colon {\mathcal U} \to
\R^{(2+\ell)|\cA|} \times C_{\perp}^{\infty}(\Delta),$ defined in a
small neighbourhood $\mathcal U$ of $(\lambda_0, 0) \in
\R^{(2+\ell)|\cA|} \times C^{\infty}_{\perp}(\Delta)$ by
\[
\Psi(\lambda, f) = \Big(\lambda, ({\rm Id} - \Pi_0)(S_z(Q_{\lambda}(\tilde
\omega_W)) \Big),
\]
where $\tilde \omega_W= \omega_W^0 + dd^c_W f$ is a compatible metric on
$\cM^{\rm comp}_{\Omega}(W)$. Note that if $f$ has sufficiently small
$C^1$-norm, the equation $({\rm Id} - \Pi_0)\circ (S_z(Q_{\lambda}(\tilde
\omega_W))=0$ is satisfied if and only if $Q_{\lambda}(\tilde \omega_W)=0$:
this follows from the fact that $\Pi_0 \circ S_{\tilde z} \circ S_{z}^{-1}$ defines a continuous family of linear endomorphisms of the finite
dimensional space of affine functions on ${\mathfrak t }^*$, with the identity
corresponding to $\tilde \omega_W= \omega_W^0$; thus $\Pi_0 \circ S_{\tilde z}
\circ S_{z}^{-1} \circ \Pi_0$ is invertible for $\tilde \omega_W$ close to
$\omega_W^0$, and hence (by using that $\Pi_0(S_{\tilde z} (Q(\tilde
\omega_W))=0$) we get
\begin{equation*}
\begin{split}
\Pi_0 \circ S_{\tilde z} \circ S_{z}^{-1} \circ ({\rm Id} - \Pi_0)\Big(S_z
(Q_{\lambda}(\tilde \omega_W)\Big) = - \Pi_0 \circ S_{\tilde z} \circ
S_{z}^{-1}\circ \Pi_0 \Big(S_z (Q_{\lambda}(\tilde \omega_W))\Big)
\end{split}
\end{equation*}
which is zero iff $S_z(Q_{\lambda}(\tilde \omega_W)=0$
i.e. $Q_{\lambda}(\tilde \omega_W)=0$.

\smallskip
By the discussion above, we are in position to complete the proof of
Proposition~\ref{perturbation} by applying the inverse function theorem to the
extension of $\Psi$ to suitable Sobolev spaces, together with elliptic
regularity (as in \cite{Le-Sim1}, see also the proof of Lemma~\ref{stability})
in order to find a family $\tilde \omega_W^\lambda= \omega_W^0 + dd^c_W
f_{\lambda}$ of smooth compatible metrics satisfying $\Psi(\lambda,
f_\lambda)= (\lambda, 0)$ for $|\lambda-\lambda_0|<\varepsilon$.

Let us first introduce the functional spaces we will work on. Recall that
${C}^{\infty}(\Delta)$ is seen as a (closed) Fr\'echet subspace of the space
of $T$-invariant smooth functions on $W$ (resp. $G$-invariant smooth functions
on $M$) which are constant on the inverse images of the momentum map $z$ for the sub-torus $\T$.  It
follows from the description of the Lie algebra of $T$ (resp. $G$) given in
Lemma~\ref{isometry algebra} that $C_{\perp}^{\infty}(\Delta)$ is
precisely the intersection of ${C}^{\infty}(\Delta)$ with the space
${C_{\perp}^{\infty}(W)}^T$ of $T$-invariant smooth functions on $W$ which
are $L^2$-orthogonal with respect to $p_{\lambda_0}{\rm Vol}_{\omega_W^0}$ to
Killing potentials of $g_W^0$ (resp. the space ${C_{\perp}^{\infty}(M)}^G$
of $G$-invariant smooth functions on $M$ which are $L^2$-orthogonal with
respect to ${\rm Vol}_{\omega_0}$ to Killing potentials of $g_0$). We let
$ L_{\perp}^{2,k}(W,\Delta)$ (resp. $L_{\perp}^{2,k}(M, \Delta)$) be the
closure of $C_{\perp}^{\infty}(\Delta)$ with respect to the Sobolev norm
$||\cdot ||_2^k$ on $W$ for the measure $p_{\lambda_0}(z) {\rm
Vol}_{\omega_W^0}$ and riemannian metric $g^0_W$  (resp. the Sobolev norm $||\cdot ||_2^k$ on $M$ with
respect to ${\rm Vol}_{\omega_0}$ and $g_0$).  For $k \gg 1$, the Sobolev embedding
$L_{\perp}^{2,k+4}(W, \Delta) \subset C_{\perp}^{3}(\Delta)$ allows us to
extend the differential operator $\Psi$ to a $C^1$-map from a neighbourhood of
$(\lambda_0,0) \in \R^{(2+\ell)|\cA|} \times 
L_{\perp}^{2,k+4}(W,\Delta)$ into $L_{\perp}^{2,k}(W, \Delta)$, such that $\Psi
(\lambda_0, 0)=0$; furthermore, as the principal part of $Q_{\lambda}$ is
concentrated in the term ${\widetilde {Scal}}_W$, one can see that $\Psi$ is a
fourth-order quasi-elliptic operator~\cite{Le-Sim1}.

Now, in order to apply the inverse function theorem, it is enough to establish
the following
\begin{lemma}
Let $ T_0\colon C_{\perp}^{\infty}(\Delta) \to C_{\perp}^{\infty}(\Delta)$
be the linearization at $\omega_W^0\in \cM^{\rm comp}_{\Omega}(W)$ of
$Q_{\lambda_0}$. Then $T_0$ is an isomorphism of Fr\'echet spaces.
\end{lemma}
\begin{proof}  Let $(g_0, J_0, \omega_0)$ be the compatible extremal K\"ahler
metric on $M$ corresponding to the initial value $\lambda=\lambda_0$.  For any
function $f\in C_{\perp}^{\infty}(\Delta)$ we consider the compatible
K\"ahler metric $\tilde g$ on $M$, with K\"ahler form $\tilde \omega= \omega_0
+ dd^c_M f$ and the compatible K\"ahler metric $\tilde g_W$ on $W$ with
K\"ahler form $\tilde \omega_W = \omega_W^0 + dd^c_W f$. We saw already in
Sect.~\ref{extremal vector field} that for $\lambda= \lambda_0$,
$Q_{\lambda_0}(\tilde \omega_W) = P_{\lambda_0}({\tilde U})= Scal_{\tilde g}^{\perp}$, where $\tilde U$ and
$Scal_{\tilde g}^{\perp}$ are the symplectic potential and normalized scalar curvature of $\tilde g$. It
then follows from \cite{gauduchon-book,Le-Sim} that the linearization $T_0$ of
$Q_{\lambda_0}$ (at $\omega_W^0$) is equal to $-2$ times the Lichnerowicz
operator $L$ of $(g_0, \omega_0)$ acting on the space of pullbacks (via $z$)
of functions in $C_{\perp}^{\infty}(\Delta)$. We have already observed in
the proof of Lemma~\ref{stability} that $L$ is an isomorphism when restricted
to the space ${C_{\perp}^{\infty}(M)}^G$ of $G$-invariant smooth functions
$L^2$-orthogonal to Killing potentials of $g_0$.  The main point here is to
refine this by showing that $L$ is an isomorphism when restricted to subspace
$C_{\perp}^{\infty}(\Delta)$, the only missing piece being the
surjectivity.

Suppose for a contradiction that $ L\colon C_{\perp}^{\infty}(\Delta) \to
C_{\perp}^{\infty}(\Delta)$ is not surjective. Considering the extension
of ${L}$ to an operator between the Sobolev spaces $L_{\perp}^{2,4}(M,\Delta) \to L_{\perp}^{2}(M,\Delta)$ (by elliptic theory $L$ is
a closed operator), our assumption is then equivalent to the existence of a
non-zero function $u \in L_{\perp}^{2}(M, \Delta)$ such that, for any
$\phi \in C_{\perp}^{\infty}(\Delta)$, $L(\phi)$ is $L^2$ orthogonal to
$u$. As any sequence of functions converging in $L^2(M)$ has a point-wise
converging subsequence, $u=u(z)$ is (the pullback to $M$ of) a
$L^{2}$-function on $\Delta$, and using \eqref{Vol} we have
\begin{equation}\label{defining}
\int_M L(\phi) u \omega_0^m= \int_{\Delta^0} L(\phi) u(z)p(z) dv = 0.
\end{equation}
We claim that
\eqref{defining} implies
\begin{equation}\label{consequence} 
\int_M L(f) u \ \omega_0^m =0
\end{equation}
for any $f \in {C_{\perp}^{\infty}(M)}^G$. This would be a contradiction
because $L$ extends to an isomorphism between the closures ${L_{\perp}^{2,4}(M)}^G$ and ${L_{\perp}^{2}(M)}^G$ of ${C_{\perp}^{\infty}(M)}^G$
in the corresponding Sobolev spaces on $M$.

It is enough to establish \eqref{consequence} by integrating on $M^0= z^{-1}(\Delta^0)$
(which is the complement of the union of submanifolds of real codimension at
least $2$).

The Lichnerowicz operator $L$ has the following general equivalent
expression~\cite{gauduchon-book,Le-Sim}
\begin{equation}\label{lichne}
{L}(f)  =\frac{1}{2}\Delta_{g_0}^2 f + g_0(dd^c f, \rho_{g_0})
+ \frac{1}{2}g_0(df, dScal_{g_0}),
\end{equation}
where $\rho_{g_0}$ is the Ricci form of $(g_0,J_0)$ and $\Delta_{g_0}$ is its
laplacian. We will use the specific form \eqref{M} of $g_0$ to express the
r.h.s of the above equality in terms of the geometry of $(V, g_V^0)$ and
$(\hat S, g_{\hat S})$.

Let $f$ be any $G$-invariant (and hence $\T$-invariant) smooth function on
$M$. It can be written on $M^0$ as a smooth function depending on $z$ and
$\hat S$ and, for any $s \in \hat S$, we will denote by $f_s(z) = f(z,s)$ the
corresponding smooth function of $z$ (Note that, as the pullback of $f$ to
$\hat M$ is smooth, $f_s(z)$ is a smooth function on $\Delta$, not only on
$\Delta^0$.) Similarly, for any $z\in \Delta^0$, $f_z(s)=f(z,s)$ stands for
the corresponding smooth function on $\hat S$.

Using \cite[Prop.~7]{hfkg2} and the specific form \eqref{M} of $g_0$, it is
straightforward to check that on $M^0$ we have
\begin{equation*}
\begin{split}
dd^c f = & \sum_{k,r,t=1}^{\ell} \frac{\partial }{\partial z_k}\Big(\frac{\partial f}{\partial z_t} H_{rt}\Big) dz_k \wedge \hat \theta_r + \sum_{j=1}^N \Big(\sum_{r,t=1}^{\ell}\frac{\partial f}{\partial z_t}H_{rt}p_{jr}\Big)\omega_j \\
& + \sum_{r=1}^{\ell} \Big(d_{\hat S} \Big( \sum_{s=1}^{\ell} \frac{\partial f}{\partial z_t} H_{rt} \Big) \wedge \hat \theta_r + d^c_{\hat S}\Big(\sum_{t=1}^{\ell} \frac{\partial f}{\partial z_t} H_{rt} \Big)  \wedge J \hat \theta_r  \Big) \\
& +  d_{\hat S} d^c_{\hat S} f_z  \\
\Delta_{g_0}  f  = & \Delta_{\hat S, z} f_z  + \Delta_{g_0} f_s ; \\
\rho_{g_0} = & \sum_{j=1}^N \rho_j  - \sum_{k,r,t=1}^{\ell} \frac{\partial}{\partial z_k} \Big( \frac{1}{2p(z)}  \frac{\partial (p(z)H_{tr})}{\partial z_t}\Big) dz_k \wedge \hat \theta_r \\ & \ \ -  \frac{1}{2p(z)} \sum_{j=1}^N\Big(\sum_{r,t=1}^{\ell} \frac{\partial (p(z)H_{rt})}{\partial z_t}p_{jr}\Big)\omega_j,\\
Scal_{g_0} =&  Scal_{\hat S, z} -\frac 1{p(z)}\sum_{r,t=1}^{\ell}
\frac{\del^2}{\del z_r\del z_t} (p(z) H_{rt}),\\
\end{split}
\end{equation*}
where 
\begin{bulletlist}
\item $p(z)=\prod_{j=1}^N (\ip{p_j,z}+c_j)^{d_j}$;
\item $\hat \theta = (\hat \theta_1, \ldots, \hat \theta_{\ell})$ and $p_j =
(p_{j1}, \ldots, p_{j\ell})$ with respect to the chosen basis of ${\mathfrak
t}$;
\item $d_{\hat S}$ and $d^c_{\hat S}$ are the differential and the
$d^c$-operator acting on functions and forms on $\hat S$;
\item $(g_j,\omega_j)$ are the product CSC K\"ahler factors of the K\"ahler
metric $(g_{\hat S}, \omega_{\hat S})$, with respective Ricci forms $\rho_j$
and laplacians $\Delta_{g_j}$;
\item $g_{\hat S, z} = \sum_{j=1}^N (\langle p_j, z) + c_j) g_j$ is the
quotient K\"ahler metric on $\hat S$ at $z$, and $\omega_{\hat S, z}$,
$Scal_{\hat S, z}$ and $\Delta_{\hat S, z}$ denote its K\"ahler form, scalar
curvature and laplacian, respectively;
\end{bulletlist}
Substituting back in \eqref{lichne}, we obtain
\begin{equation*}
\begin{split}
{L} (f) =  & {L} (f_s) + {L}_{\hat S,z} (f_z) + \Delta_{\hat S, z}\big((\Delta_{g_0} f_s)_z\big) + \Delta_{g_0} \big( (\Delta_{\hat S, z} f_z)_s\big)  \\ 
& + \sum_{j=1}^N R_j(z)\Delta_{g_j} (f_z), 
\end{split}
\end{equation*}
where $L_{\hat S, z}$ is the Lichnerowicz operator of $g_{\hat S, z}$, and
$R_j(z)$ are coefficients (that can be found explicitly from the above
formulae) depending only on $z$, and such that $p(z)R_j(z)$ are smooth on
$\Delta$.

If we integrate the above expression for $L(f)$ against $u(z)$ (by using
\eqref{Vol}) we get that $\int_M L(f) u \ \omega_0^m$ is a non-zero constant
multiple of
\begin{equation*}
\begin{split}
&  \int_{\hat S} \Big(\int_{\Delta^0} {L}(f_s) u(z)p(z)dv\Big) \omega_{\hat S}^{d}
+ \int_{\hat S} \Big( \int_{\Delta^0}  \Delta_{g_0}\big((\Delta_{\hat S, z} f_z)_s\big) u(z) p(z) dv \Big) \omega_{\hat S}^d  \\
&  + \int_{\Delta^0} \Big( \int_{\hat S} {L}_{\hat S, z} (f_z)
\omega_{\hat S,z}^{d}\Big)u(z) dv
+ \int_{\Delta^0} \Big( \int_{\hat S} {\Delta}_{\hat S, z} \big((\Delta_{g_0} f_s)_z\big) \omega_{\hat S,z}^{d}\Big)u(z)dv \\
& +  \sum_{j=1}^N  \int_{\Delta^0}\Big(\int_{\hat S} {\Delta}_{g_j} (f_z)  \ \omega_{\hat S}^{d}\Big) p(z)R_j(z)u(z) dv.
\end{split}
\end{equation*} 
To see that all the terms vanish, note that the first term is zero by
\eqref{defining}; the third and fourth terms are zero because ${L}_{\hat S, z}$
and $\Delta_{\hat S, z}$ are self-adjoint (with respect to $\omega_{\hat S,
z}$) and therefore their images are $L^2$-orthogonal to constants on $\hat
S$. The fifth term is also zero because $\Delta_{g_j} (f)$ is $L^2$-orthogonal
to constants on $\hat S$ with respect to $\omega_{\hat S}$: this follows
easily from the local product structure of $g_{\hat S}$. For the second term
one uses that $\Delta_{g_0}$ defines a self-adjoint operator on
$C^{\infty}(\Delta)$ with respect to the measure $p(z)dv$: thus, for any
smooth function $\phi(z)$ on $\Delta$,
\[
\int_{\hat S} \Big( \int_{\Delta^0} \Delta_{g_0}\big( (\Delta_{\hat S, z}
f_z)_s\big) \phi(z)p(z) dv \Big) \omega_{\hat S}^d =\int_{\Delta^0}
\Big(\int_{\hat S} \Delta_{\hat S, z} (f_z) \omega_{\hat S, z}^d\Big)
(\Delta_{g_0}\phi) dv=0
\]
because $\Delta_{\hat S, z} f_z$ is $L^2$-orthogonal to constants on $\hat S$;
as $u$ is in the closure in $L^2$ of pullbacks of smooth functions on
$\Delta$, the second term vanishes too.

This concludes the proof of the lemma.
\end{proof}

An immediate consequence of Proposition~\ref{perturbation} is the following
\begin{cor}\label{corollary1} The existence of a compatible extremal K\"ahler metric is an open
condition on the set of admissible K\"ahler classes on $M$. 
\end{cor}
\begin{proof}  As we have already observed, the admissible K\"ahler classes
are parametrized by the real constants $c_a$ for $a \in \cA$.  We thus apply
Proposition~\ref{perturbation} by taking $\lambda = (c_a, p_a^0,
Scal_a^0)$.
\end{proof}

\subsection{Proof of Theorem \ref{th:small-classes}}\label{s:existence}

To deduce Theorem~\ref{th:small-classes} from
  Proposition~\ref{perturbation}, 
we observe that
the differential operators \eqref{P-lambda} satisfy $P_{t \lambda } =
P_{\lambda}$ for any real number $t \neq 0$.

On any K\"ahler manifold $(M, g, \omega)$ obtained by the generalized Calabi
construction with data $\lambda=(c_a, p_a, Scal_a)$, we can consider the
sequence of differential operators $P_{\lambda_k}$ where $\lambda_k =
(c_a+k, p_a, Scal_a)$. The differential
operator $P_{\lambda_k}$ is the same as $P_{\frac{\lambda_k}{k}}$ and
$\frac{\lambda_k} {k}$ converges when $k\to \infty$ to the data corresponding
to the extremal K\"ahler metric equation for a compatible K\"ahler metrics on
$W$.  We then readily infer Theorem \ref{th:small-classes} from 
Proposition~\ref{perturbation}.
\begin{rem}\label{trivial-corollary}
As any invariant K\"ahler metric on a toric manifold is compatible,  Theorem~\ref{th:small-classes}  implies the existence of (compatible) extremal metrics on a rigid semisimple toric bundles $M$ over a CSC locally product K\"ahler manifold,  in the case  when there are no blow-downs and $W=V$ is a toric
extremal K\"ahler manifold.
\end{rem}

\begin{rem}\label{multiplicity-free} An interesting class of rigid toric
bundles comes from the theory of multiplicity-free manifolds recently
discussed in \cite{Do6}.  A typical example is obtained by taking a compact
connected semisimple Lee group $G$ and a maximal torus $\T \subset G$ with
Lie algebra ${\mathfrak t}$; if we pick a positive Weyl chamber ${\mathfrak
t}_+ \subset {\mathfrak t}$ (and identify ${\mathfrak t}$ with its dual
space ${\mathfrak t}^*$ via the Killing form), for any Delzant polytope
$\Delta$ contained in the interior of ${\mathfrak t}_+$, one can consider
the manifold $M = p: G \times_{\T} V \to S=G/\T$, where $V$ is the toric
manifold with Delzant polytope $\Delta$.  Note that $G$ has a structure of
principal $\T$-bundle over the flag manifold $S=G/\T$ with a connection
$1$-form $\theta \in \Omega^1(G, {\mathfrak t})$ whose curvature
$\omega(z)=\langle d\theta, z\rangle$ defines a family of symplectic forms
on $S$ (the Kirillov--Kostant--Souriau forms); identifying $S \cong G^c/B$,
where $B$ is a Borel subgroup of the complexification $G^c$ of $G$, each
$\omega(z)$ defines a homogeneous K\"ahler metric $g(z)$ on the complex
manifold $S$ (which is therefore of constant scalar curvature); the Ricci
form $\omega_S$ of $\omega(z)$ is independent of $z$, giving rise to the
normal (K\"ahler--Einstein) metric $g_S$ on $S$. Now, for any toric
K\"ahler metric on $V$, corresponding to a symplectic potential $U\in
{\mathcal S}(\Delta)$, one considers the K\"ahler metric on $M$
\[
g=p^*(g(z) +k g_S) + \langle dz, {\bf G}, dz \rangle + \langle \theta,
{\bf H}, \theta \rangle, \ \ \omega= p^* (\omega(z) + k \omega_S) + \langle
dz \wedge d\theta \rangle,
\]
where ${\bf G} = {\rm Hess}(U)$, ${\bf H} = {\bf G}^{-1}$, $z \in \Delta$
and $k>0$.  In this case, $G \to S=G/\T$ is not necessarily a {\it
diagonalizable} principal $\T$-bundle over $S=G/\T$ (in other words, $M= G
\times_{\T} V \to S=G/\T$ is a rigid but not in general semisimple toric
bundle). However, most parts of the discussion in Sect.~\ref{calabi-type}
do extend to this case too (see also \cite{hfkg1}), with some obvious
modifications.  The key points are that (a) the volume form of $g(z) +k
g_S$ is a multiple $p(z)$ (depending only on $z$) of ${\rm Vol}_{g_S}$:
this allows to extend the curvature computations (see
\cite[Prop.~7]{hfkg1}) and formula \eqref{Vol} to this case, (b) for any $z
\in \Delta$, $g(z) +k g_S$ is a CSC K\"ahler metric on $S$: this allows to
extend the results in Sect.~\ref{isometry}, and (c) there is a similar
formula to \eqref{scal} for the scalar curvature of $g$, found by
Raza~\cite{raza}, which allows to reduce the extremal equation for the
K\"ahler metrics in the above form to \eqref{operator} with $p_a$ being
essentially the positive roots of $G$, $c_a=k$ and $Scal_a$ positive
constants.  Proposition~\ref{perturbation} and its corollaries
(Corollary~\ref{corollary1} and Theorem~\ref{th:small-classes}) extend to
this setting too. We thus get both openness and existence of extremal
K\"ahler metrics of the above form when $V$ is an extremal toric K\"ahler
variety and $k\gg 0$.
\end{rem}

\section{Proof of Theorem~\ref{th:existence}}\label{sec:proof-existence} As
another application of Theorem~\ref{th:small-classes}, we derive
Theorem~\ref{th:existence} from the introduction. This is the case when $V=
\C P^{\ell}$ and $W= \C P^r, \ r \ge \ell\ge 1$ and $M = P(E_0 \oplus
\cdots \oplus E_{\ell}) \to S$ (see Sect.~\ref{s:projective bundles}). It
follows from the general theory of hamiltonian $2$-forms \cite{hfkg1,hfkg2}
that any Fubini--Study metric on $\C P^r$ admits a rigid semisimple
isometric action of an $\ell$-dimensional torus $\T$, for any $1 \le \ell
\le r$ (see in particular \cite[Prop.~17]{hfkg1} and \cite[Thm.~5]{hfkg2}):
thus, $W=\C P^{r}$ admits a compatible extremal K\"ahler metric.

Let $\omega$ be a compatible K\"ahler on $M$; as the fibre is $\C P^r$, by
re-scaling, we can assume without loss that $[\omega]= 2\pi c_1(\cO(1)_E) +
p^*\alpha$, where $\alpha$ is a cohomology class on $S$. The form \eqref{M} of
$\omega$ and the assumption on the first Chern classes $c_1(E_i)$ imply that
$\alpha$ is diagonal with respect to the product structure of $S$, in the
sense that it pulls back to the covering product space as $\alpha= \sum_{a \in
\cA} q_a [\omega_a]$ for some real constants $q_a$. Therefore, $\Omega_k =
2\pi c_1(\cO(1)_E) + k p^* [\omega_S] =[\omega] + \sum_{a\in \cA}
(k-q_a)p^*[\omega_a]$.  If we choose $q$ with $q>q_a$, then $\tilde \omega
= \omega + \sum_{a\in \cA} (q -q_a)p^*\omega_a$ is clearly a compatible
K\"ahler metric too. Thus, $\Omega_k = [\tilde \omega] + (k-q)p^*[\omega_S]$
with $[\tilde \omega]$ compatible, and we derive Theorem~\ref{th:existence}
from the introduction as a particular case of Theorem~\ref{th:small-classes}.

\section{Proof of Theorem~\ref{th:extremal}} \label{partial}

Suppose that $(g, \omega)$ is an extremal K\"ahler metric in $\Omega_k=2\pi
c_1(\cO(1)_E) + k p^* [\omega_{\Sigma}]$ on $(M,J)= P(E_0 \oplus \ldots \oplus
E_{\ell}) \to \Sigma$, where $E_i$ are indecomposable holomorphic vector
bundles over a compact curve $\Sigma$ of genus ${\bf g}\ge 2$.  We can assume
without loss that $\omega_{\Sigma}$ is the K\"ahler form of a constant
curvature metric on $\Sigma$ and, by virtue of Theorem~\ref{main}, that the
scalar curvature of $g$ is not constant. In particular, $\ell \ge 1$.

We have seen in Lemma~\ref{decompose} that the $\ell$-dimensional torus $\T$
acting by scalar multiplication on each $E_i$ is maximal in the reduced
automorphism group ${\widetilde {\rm Aut}}_0(M,J) \cong H^0(\Sigma,
PGL(E))$. By a well-known result of Calabi~\cite{cal-two} the identity
component of the group of K\"ahler isometries of an extremal K\"ahler metric
is a maximal compact subgroup of ${\rm Aut}_0(M,J)$, so we can assume without
loss that $(g, \omega)$ is $\T$-invariant.

By considering small stable deformations $E_i(t)$ and applying
Lemma~\ref{stability}, we can find a smooth family of extremal $\T$-invariant
K\"ahler metrics $(J_t, g_t, \omega_t)$, converging to $(J,\omega)$ in any
$C^{k}(M)$, such that $(M,J_t) \cong P(\bigoplus_{i=1}^{\ell} E_i(t)) $, and
$[\omega_t] = [\omega]$ in $H^2_{dR}(M)$. By the equivariant Moser lemma, we
can assume without loss that $\omega_t = \omega$.

It is not difficult to see that any K\"ahler class on $(M,J_t)$ (for $t \neq
0$) is {\it compatible}: this follows from the fact that the cohomology
$H^2(M) \cong H^{1,1}(M,J_t)$ is generated by any compatible K\"ahler class on
$(M,J_t)$ and the pullback $p^* [\omega_{\Sigma}]$.  By
Theorem~\ref{th:existence} and the uniqueness of the extremal K\"ahler metrics
up to automorphisms~\cite{CT}, for any $t \neq 0$ we can take $k \gg 0$ such
that the extremal K\"ahler metric $(g_t, \omega)$ on $(M,J_t)$ is compatible
with respect to the rigid semisimple action of the maximal torus
$\T$. Strictly speaking, Theorem~\ref{th:existence} produces a lower bound
$k_0$ for such $k$, depending on $J_t$. However, in our case $|\cA| =1$,
the simplex $\Delta$, the moment map $z$ and the metric on $\Sigma$ are fixed,
and the parameter $\lambda=(c, p, Scal_{\Sigma})$ defining the corresponding
extremal equation \eqref{operator} for a compatible metric on $(M,J_t,
[\omega])$ is independent of $t$: indeed, the constants $p \in \mathfrak t$
and $c\in \R$ are determined by the first Chern classes $c_1(E_i)$ and the
cohomology class $\Omega_k = [\omega] \in H_{dR}^2(M)$. Thus, the deformation
argument used in Sect.~\ref{s:existence} produces a lower bound $k_0$
independent of $t$, such that for any $k>k_0$ and $t \neq 0$, $(g_t, \omega)$
is an extremal K\"ahler metric in $\Omega_k$ with respect to which the maximal
torus $\T$ acts in a rigid and semisimple way.

Take a regular value $z_0$ of the momentum map $z$ associated to the
hamiltonian action of $\T$ on $(M,\omega)$ and consider the family of
K\"ahler quotient metrics $(\hat{g}_{t}, \hat{J}_t)$ on the symplectic
quotient $\hat S$. By identifying the symplectic quotient with the stable
quotient, we see that $(\hat S, \hat{J}_t) \cong
P(\vE_0(t))\times_{\Sigma}\cdots\times_{\Sigma} P(\vE_\ell(t))\to \Sigma$
(see Sect.~\ref{s:projective bundles}). As for $t\neq 0$ the action of $\T$
is rigid and semisimple and $g_t$ is compatible, the quotient K\"ahler
metric $(\hat{g}_t, \hat{J}_t)$ must be locally a product of CSC K\"ahler
metrics. By the de Rham decomposition theorem $\hat{g}_t$ must be a
locally-symmetric metric modelled on the hermitian-symmetric space $\C
P^{d_0} \times \cdots \times \C P^{d_{\ell}} \times {\mathbb H}$, where
$d_i +1 = {\rm rk}(E_i)$ (so that $\C P^{d_i}$ is a point if $d_i=1$) and
${\mathbb H}$ is the hyperbolic plane. By continuity, $(\hat{g}_0,
\hat{J}_0)$ is a locally-symmetric K\"ahler metric on $\hat S$ of the same
type. By the de Rham decomposition theorem and considering the form of the
covering transformations we obtain representations $\rho_i\colon \pi_1
(\Sigma) \to PU(d_i+1)$, and therefore $E_i$ must be stable by the standard
theory~\cite{NS}.

\smallskip

In the case when $\ell=1$, we can assume without loss by Theorem~\ref{main}
that $E$ is not polystable, and we can then use instead of
Theorem~\ref{th:existence} the stronger results~\cite[Thm.~1~\& 6]{ACGT} which
affirm that {\it any} extremal K\"ahler metric on $(M, J_t)$ (for $t \neq 0$)
must be compatible with respect to the natural $S^1$-action.

\section{Further observations}\label{discussion}

\subsection{Relative K-energy and the main conjecture}

Leaving aside the specific motivation of this paper to study projective
bundles over a curve, the theory of rigid semisimple toric bundles which we
reviewed in Sect.~\ref{calabi-type} extends the theory of extremal K\"ahler metrics on toric
manifolds~\cite{Do2,Do3,Do5,Sz,ZZ,ZZ2} to this more general context.

To recast the leading conjectures~\cite{Do2,Sz} in the toric case to this
setting, recall from \cite{Do2} that if we parametrize compatible K\"ahler
metrics $g$ by their symplectic potentials $U \in \cS(\Delta)$, then the
relative (Mabuchi--Guan--Simanca) K-energy $\cE^\Omega$ on this space
satisfies the functional equation
\begin{align*}
(d\cE^\Omega)_g(\dot U) &= \int_\Delta (\Scal_g^{\perp}) \dot U(z)
p(z) dv\\ &=\int_\Delta\biggl( \Bigl( \ip{A,z}+B+\sum_{j=1}^N
\frac{\Scal_j}{\ip{p_j,z}+c_j}\Bigr) p(z) - \frac{\del^2}{\del z_r\del z_s}
(p(z) H_{rs})\biggr) \dot U(z) dv\\ &= 2\int_{\del\Delta} \dot U(z) p(z)
d\sigma+ \int_\Delta\Bigl( \ip{A,z}+B+\sum_{j=1}^N
\frac{\Scal_j}{\ip{p_j,z}+c_j}\Bigr) \dot U(z) p(z) dv\\ &\quad-\int_\Delta
\ip{{\bf H}, \mathrm{Hess}\,\dot U(z)} p(z) dv,
\end{align*}
where we have used \eqref{operator} and integration by parts by taking into account  \eqref{eq:toricboundary}.   Following \cite{Do2,Sz,ZZ}, 
let us introduce the linear functional
\begin{equation}\label{F}
\cF^\Omega(f):=\int_{\del\Delta} f(z) p(z) d\sigma
+\frac12 \int_\Delta
\Bigl( \ip{A,z}+B+\sum_j \frac{\Scal_j}{\ip{p_j,z}+c_j}\Bigr)f(z) p(z) dv.
\end{equation}
The above calculation of $d\cE^{\Omega}_g$ shows that $\cF^{\Omega}(f)=0$ if $f$ is an affine function of $z$. Furthermore, using the fact that the derivative of $\log\det {\bf H}$ is $\trace {\bf H}^{-1}d{\bf H}$, we obtain the following generalization of Donaldson's formula for $\cE^\Omega$:
\begin{equation}\label{K-energy}
\cE^\Omega(U) = 2\cF^\Omega(U)
-\int_\Delta \Big(\log\det\mathrm{Hess}\, U(z)\Big)p(z) dv.
\end{equation}
(In case of doubt about the convergence of the integrals, one can introduce a
reference potential $U_c$ and a relative version $\cE^\Omega_{g_c}$ of
$\cE^\Omega$, but in fact, as Donaldson shows, the convexity of $U$ ensures
that the positive part of $\log\det\mathrm{Hess}\, U(z)$ is integrable, hence
$-\log\det\mathrm{Hess}\, U(z)$ has a well defined integral in
$(-\infty,\infty]$.)

According to~\cite{Do2,Sz}, the existence of a solution $U \in \cS(\Delta)$ to \eqref{operator} should be entirely governed by properties of the linear functional \eqref{F}:

\begin{conj}\label{con:2} Let $\Omega$ be a compatible class on $M$. Then the following conditions  should be equivalent:
\begin{enumerate}
\item $\Omega$ admits an extremal K\"ahler metric.
\item $\Omega$ admits a compatible extremal K\"ahler metric
(i.e.~\eqref{operator} has a solution in $\cS(\Delta)$).
\item $\cF^\Omega(f) \ge 0$ for any piecewise linear convex function $f$ on
$\Delta$, and is equal to zero if and only if $f$ is
affine.\footnote{Generalizing computations in \cite{Do2,Sz,ZZ}, one can
show that the value of $\cF^\Omega$ at a {\it rational} piecewise linear
convex function computes the relative Futaki invariant introduced in
\cite{Sz} of a `compatible' toric test configuration on $(M,\Omega)$; in
general, one might need positivity of $\cF^\Omega$ on a larger space of
convex functions~\cite[Conjecture~7.2.2.]{Do2} in order to solve
\eqref{operator} but in the case when $\ell=2$ and the base $\hat S$ is a
point Donaldson shows in \cite{Do2} that the space of piecewise linear
convex functions will do.}
\end{enumerate}
\end{conj}

Of course, by the proof of Theorem~\ref{th:extremal},
Conjecture~\ref{con:2} would imply Conjecture~\ref{con:1}.

Our formula \eqref{K-energy} can be used to show as
in~\cite[Prop.~7.1.3]{Do2} that $\cF^{\Omega}(f) \ge 0$ if the relative
K-energy is bounded from below. However, according to Chen--Tian~\cite{CT},
the boundedness from below of $\cE^\Omega$ is a necessary condition for the
existence of an extremal K\"ahler metric.

If $\Omega$ admits a {\it compatible} extremal K\"ahler metric with symplectic
potential $U$ and inverse hessian ${\bf H}$, one can use \eqref{operator} and
integration by parts (taking into account \eqref{eq:toricboundary}) in order
to re-write \eqref{F} as
\begin{equation}\label{F-extremal}
\cF^{\Omega} (f) = \int_{\Delta} \langle {\bf H}, {\rm Hess} f \rangle p(z)
dv.
\end{equation}
This formula makes sense for smooth functions $f(z)$, but can also be used to
calculate $\cF^{\Omega} (f)$ in distributional sense for any piecewise linear
convex function as in \cite{ZZ2}: using the fact that ${\bf H}$ is positive
definite, we obtain the analogue of a result in \cite{ZZ2}, showing that the
second statement of Conjecture~\ref{con:2} implies the third.

We thus have the following partial result.
\begin{prop}\label{p:stability} If $\Omega$ admits an extremal K\"ahler metric
then $\cF^{\Omega}(f) \ge 0$ for any convex piecewise linear function. If
$\Omega$ admits a compatible extremal metric then, furthermore,
$\cF^{\Omega}(f)=0$ if and only if $f$ is an affine function on $\Delta$.
\end{prop} 
Of course, the most difficult part of Conjecture~\ref{con:2} is to prove (3)
$\Rightarrow$ (2). So far the Conjecture~\ref{con:2} has been fully
established in the cases when $\ell=1$~\cite{ACGT} and when $M$ is a toric
surface (i.e. $\ell=2$ and $\hat S$ is a point) with vanishing extremal vector
field~\cite{Do5}.

\subsection{Computing $\cF^{\Omega}$}

It is natural to consider (following Donaldson~\cite{Do2}) the space of $S^2\mathfrak t^*$-valued functions ${\bf H}$ on
$\Delta$ satisfying just the boundary conditions \eqref{eq:toricboundary}. If
such a function satisfies the (underdetermined, linear) equation
\eqref{operator}, then formula \eqref{F-extremal} holds, and it can be used to
compute the action of $\cF^{\Omega}$ (in distributional sense) on piecewise
linear functions.

Note that if a solution to \eqref{operator} exists, then so do many because
the double divergence is underdetermined. 

% Indeed, there are locally exact adjoint complexes of linear differential operators on $\Delta\subset\mathfrakt^*$: 
%\begin{equation*}\label{sequences}\begin{split} C^\infty(\R)&\rightarrow C^\infty(S^2\mathfrak t)\rightarrow C^\infty(\Lambda^2\mathfrak t\odot\mathfrak t)\rightarrow\cdots\\ C^\infty(\R)&\leftarrow C^\infty(S^2\mathfrak t^*)\leftarrow C^\infty(\Lambda^2\mathfrak t^*\odot\mathfrak t^*)\leftarrow\cdots, \end{split}\end{equation*} where $\Lambda^2\mathfrak t\odot\mathfrak t$ denotes the alternating-free tensors in $\Lambda^2\mathfrak t\otimes\mathfrak t$ (the kernel of the projection, alternation, to $\Lambda^3\mathfrak t$). The first two arrows in the top line are the hessian and the exterior derivative (of a $\mathfrak t$ valued $1$-form). The adjoint operators in the bottom line are the double divergence and the symmetrized divergence.

If a solution ${\bf H}$ of \eqref{operator} happens to be {\it positive
definite} on each face of $\Delta$, i.e. if it verifies the positivity
condition in Sect.~\ref{s:generalised-calabi}, then formulae \eqref{M}
introduce an almost K\"ahler metric on $M$ (see e.g.~\cite{hfkg2}) and one can
show that \eqref{scal} computes its hermitian scalar curvature (see
Appendix~\ref{app:almost-kahler}). Thus, positive definite solutions of
\eqref{operator} correspond to compatible extremal almost K\"ahler metrics.
If such extremal almost K\"ahler metrics exist, it then follows from \eqref{F-extremal} (see~\cite{ZZ2} and
Proposition~\ref{p:stability} above) that the condition (3) of
Conjecture~\ref{con:2} is verified. Thus, the existence of a positive definite
solution ${\bf H}$ of \eqref{operator} (and verifying the boundary conditions
\eqref{eq:toricboundary}) is conjecturally equivalent to the existence of a
compatible extremal K\"ahler metric (corresponding to another positive
definite function ${\bf H}^{\Omega}$ with inverse equal to the hessian of a
function $U_{\Omega}$). In fact, following \cite{Do2}, as $\log\det$ is
strictly convex on positive definite matrices, the functional $\int_\Delta
(\log \det{\bf H}) p(z) dv$ is strictly convex on the space of positive
definite solutions of \eqref{operator}, and therefore has at most one minimum
${\bf H}^\Omega$.  Such a minimum would automatically have its inverse equal
to the hessian of a function $U_{\Omega}$ (see  \cite{Do2}).  Thus, ${\bf H}^\Omega$ would
then give the extremal K\"ahler metric in the compatible K\"ahler class
$\Omega$.

Thus motivated, it is natural to wonder if on the manifolds we consider in
this paper a (not necessarily positive definite) solution ${\bf H}$ of
\eqref{operator} exists, thus generalizing the extremal polynomial
introduced in \cite{ACGT} on $M=P(E_0 \oplus E_1)\to S$ (in fact ${\bf
P}(z)=p(z){\bf H}(z)$ would be the precise generalization).

\subsection{Example: projective plane bundles over a curve}

We illustrate the above discussion by explicit calculations on the manifold
$M= P(\cO \oplus \cL_1 \oplus \cL_2) \to \Sigma$, where ${\cL}_1$ and
${\cL}_2$ are holomorphic line bundles over a compact complex curve $\Sigma$
of genus ${\bf g}$. We put $p_i= {\rm deg}(\cL_i)$ and assume without loss
that $p_2\ge p_1\ge0$. Note that in the case $p_1=p_2=0$, the vector bundle
$E= \cO \oplus \cL_1 \oplus \cL_2$ is polystable, and therefore the existence
of extremal K\"ahler metrics is given by Theorem~\ref{main}. The cases
$p_1=p_2>0$ and $p_2>p_1=0$, on the other hand, are solved in \cite{ACGT}. We
thus assume furthermore that $p_2>p_1 >0$.

To recast our example in the set up of Sect.~\ref{calabi-type}, we take a
riemannian metric $g_{\Sigma}$ of constant scalar curvature $4(1-{\bf g})$ on
$\Sigma$. To ease the notation, we put $C= 4({\bf g} -1)$. Let $z_{i}$ be the
momentum map of the natural $S^1$-action by multiplication on $\cL_i$. Thus,
without loss, for a compatible K\"ahler metric on $M$, the momentum coordinate
$z=(z_{1},z_{2})$ takes values in the simplex $\Delta = \{ (z_{1}, z_{2}) \in
{\mathbb R}^{2}\, | \, z_{1} \geq 0, z_{2} \geq 0, 1-z_{1}-z_{2} \ge 0 \}$
(which is the Delzant polytope of the fibre $\C P^2$ viewed as a toric
variety).

It is shown in \cite[App. A2]{ACGT} that in this case there are no extremal
compatible K\"ahler metrics with a hamiltonian $2$-form of order $2$ while
Theorem~\ref{th:extremal} does imply existence of compatible extremal
K\"ahler metrics in small K\"ahler classes. Therefore, we do not have an
{\it explicit} construction of these extremal K\"ahler metrics. Instead, we
will now attempt to find explicit extremal almost K\"ahler metrics (see the
preceding section and the Appendix below). We thus want to find a smooth
matrix function ${\bf H}(z)= (H_{rs}(z))$ satisfying the boundary
conditions \eqref{eq:toricboundary} and which solves the linear equation
\eqref{operator}. Motivated by the explicit form of such a matrix in the
case when a hamiltonian $2$-form does exist~\cite{hfkg1}, we look for
solutions of a `polynomial' form $H_{rs} = \frac{P_{rs}}{(c + p_{1}z_{1} +
p_{2} z_{2})}$, where $P_{rs}(z)$ are fourth degree polynomials in $z_{1}$
and $z_{2}$, and the constant $c$ is such that $c + p_{1}z_{1} + p_{2}
z_{2}>0$ on $\Delta$ (recall that $c$ parametrizes compatible K\"ahler
classes on $M$). The boundary conditions are then solved by
\begin{align*}
P_{11} &= 2 (c + p_1 z_1 + p_2 z_2) z_1 (1 - z_1)\\ &\qquad+ 
 z_1^2 (x_0 z_2^2 + x_2 (1 - z_1 - z_2)^2 + 2 y_1 z_2 (1 - z_1 - z_2)),\\
P_{12} & = -2 (c + p_1 z_1 + p_2 z_2) z_1 z_2 \\&\qquad+ 
 z_1 z_2 (y_0 (1 - z_1 - z_2)^2 - 
    x_0 z_1 z_2 - (1 - z_1 - z_2) (y_1 z_1 + y_2 z_2)),\\
P_{22} & =  2 (c + p_1 z_1 + p_2 z_2) z_2 (1 - z_2)\\ &\qquad + 
 z_2^2 (x_0 z_1^2 + x_1 (1 - z_1 - z_2)^2 + 2 y_2 z_1 (1 - z_1 - z_2)),
\end{align*}
where $x_{0},x_{1},x_{2},y_{0},y_{1},y_{2}$ are free parameters. The extremal
condition \eqref{operator} corresponds to the linear equations
\begin{equation}\label{extr}
\begin{split}
y_{0} &= -x_1 - x_2 + v_0, \\ 
y_{1} &= -x_0 - x_2 + v_1, \\
y_{2}  &= -x_0 - x_1 + v_2, 
\end{split}
\end{equation}
with
\begin{align*}
v_{0} & = \tfrac{-(12 c + C + 4 p_1 + 4 p_2) 
(5 c p_1^2 + p_1^3 + 5 c p_1 p_2 + 5 p_1^2 p_2 + 
     5 c p_2^2 + 5 p_1 p_2^2 + p_2^3))}{(2 (50 c^3 + 50 c^2 p_1 + 
     13 c p_1^2 + p_1^3 + 50 c^2 p_2 + 37 c p_1 p_2 + 5 p_1^2 p_2 + 
     13 c p_2^2 + 5 p_1 p_2^2 + p_2^3)}\\
v_{1} & = \tfrac{-(12 c + C + 4 p_1 + 4 p_2) 
(15 c p_1^2 + 3 p_1^3 - 15 c p_1 p_2 + 
     3 p_1^2 p_2 + 5 c p_2^2 - 3 p_1 p_2^2 + p_2^3))}{(2 (50 c^3 + 
     50 c^2 p_1 + 13 c p_1^2 + p_1^3 + 50 c^2 p_2 + 37 c p_1 p_2 + 
     5 p_1^2 p_2 + 13 c p_2^2 + 5 p_1 p_2^2 + p_2^3)}\\
v_{2} & = \tfrac{-(12 c + C + 4 p_1 + 4 p_2)
(5 c p_1^2 + p_1^3 - 15 c p_1 p_2 - 
     3 p_1^2 p_2 + 15 c p_2^2 + 3 p_1 p_2^2 + 3 p_2^3))}{(2 (50 c^3 + 
     50 c^2 p_1 + 13 c p_1^2 + p_1^3 + 50 c^2 p_2 + 37 c p_1 p_2 + 
     5 p_1^2 p_2 + 13 c p_2^2 + 5 p_1 p_2^2 + p_2^3)}.
\end{align*}
Thus, given a compatible K\"ahler class on $M$, we have a $3$-parameter family
of smooth `polynomial' solutions ${\bf H}(z)$ to \eqref{operator}, which
verify the boundary conditions \eqref{eq:toricboundary} on $\Delta$. Now,
investigating the integrability condition (that ${\bf H}^{-1}$ be a hessian of
a smooth function on $\Delta^0$, see the previous section), we find out that it is equivalent to the following five algebraic equations on the parameters $(x_0,x_1,x_2, y_0, y_1.y_2)$
\begin{gather}\label{lin-int}
x_0 = y_1+y_2, \qquad
x_1 = y_2+y_0, \qquad
x_2 = y_0+y_1,
\end{gather}
\begin{equation}\label{non-lin-int}
\begin{split}
2 (p_2 - p_1) y_0 + 2 p_2 y_1 - y_0 y_1 &= 0\\
2 (p_1 - p_2) y_0 + 2 p_1 y_2 - y_0 y_2 &= 0.
\end{split}
\end{equation}
The problem is over-determined, but there is a unique solution ${\bf
H}^{\Omega}_0$ satisfying the linear system (\ref{lin-int}) (additionally to
\eqref{extr}): we compute that this solution is given by
\begin{align*}
x_0 &= \tfrac{1}{10} (-2v_0 + 3v_1 + 3v_2), \\
x_1 &= \tfrac{1}{10}(3v_0 - 2v_1 + 3v_2), \\
x_2 & = \tfrac{1}{10}(3v_0 + 3v_1 - 2v_2),\\
y_0 &=  \tfrac{1}{10} (4v_0 - v_1 - v_2), \\
y_1  &= \tfrac{1}{10}(-v_0 + 4 v_1 - v_2), \\
y_2 &= \tfrac{1}{10} (-v_0 -v_1 + 4v_2).
\end{align*}
Substituting back in \eqref{non-lin-int}, one sees that the full integrability
conditions can be solved if $12 c + C + 4 p_1 + 4 p_2 = 0$ (a constraint that is never
satisfied for $p_2> p_1 \ge 1, C = 4({\bf g}-1)$ and $c>0$); this
observation is consistent with the non-existence result in
\cite[App.~A2]{ACGT}.

\smallskip
We now investigate the positivity condition for our distinguished solution
${\bf H}^{\Omega}_0$ of \eqref{eq:toricboundary} and \eqref{operator}.
First of all, when $c\to \infty$, the $v_i$'s tend to $0$, so ${\bf
H}^{\Omega}_0$ tends to the matrix associated to a Fubini--Study metric on
$\C P^2$. It follows that ${\bf H}^{\Omega}_0$ becomes positive-definite on
each face for sufficiently small K\"ahler classes, and therefore ${\bf
H}^{\Omega}_0$ defines an {\it explicit} extremal (non-K\"ahler) almost
K\"ahler metric in $\Omega$ (see Appendix~\ref{app:almost-kahler}
below). This is of course consistent (via Conjecture~\ref{con:2}) with the
existence of a (non-explicit) extremal K\"ahler metric in $\Omega$, given
by Theorem~\ref{th:extremal}.  Furthermore, if ${\bf g} = 0,1$
(i.e. $C<0$), a computer assisted verification shows that, in fact, ${\bf
H}^{\Omega}_0$ is positive definite on each face of $\Delta$ for {\it all}
K\"ahler classes. We thus obtain the following result.

\begin{prop}\label{pr:existence} Let $M= P(E) \rTo^p \Sigma$ with
$E= \cO \oplus \cL_1 \oplus \cL_2$, where ${\cL}_1$ and ${\cL}_2$ are
holomorphic line bundles of degrees $1\le p_1<p_2$ over a compact complex
curve $\Sigma$ of genus ${\bf g}$.

If ${\bf g} = 0, 1$, then $M$ admits a compatible extremal almost K\"ahler
metric for the K\"ahler form of any compatible K\"ahler metric on $M$. In
particular, for every K\"ahler class on $M$ the condition \textup{(3)} of
Conjecture~\textup{\ref{con:2}} is verified.

If ${\bf g} \ge 2$, then the same conclusion holds for the compatible
K\"ahler forms in sufficiently small K\"ahler classes $\Omega_k = 2\pi
c_1(\cO(1)_E) + kp^*[\omega_{\Sigma}], \ k \gg 0$.
\end{prop}

\smallskip
As speculated in the previous section, the explicit solution ${\bf
H}^{\Omega}_0$ of \eqref{eq:toricboundary} and \eqref{operator} can be used to
compute the action of the functional ${\cF}^{\Omega}$ on piecewise linear
convex functions (by extending formula \eqref{F-extremal} in a distributional
sense, after integrating by parts and using \eqref{eq:toricboundary}). As a
simple illustration of this, let us take a simple crease function $f_a$ with
crease along the segment $S_a=\{(t,a-t), 0<t<a\}$ for some $a \in (0,1)$ (thus
as $a \rightarrow 0$, the crease moves to the lower left corner of the simplex
$\Delta$). A normal of the crease is $u=(1,1)$ and one easily finds that
\begin{equation}
\begin{split}
\cF^{\Omega}(f_a) &= \int_{S_a} H^{\Omega}_0(u,u) d\sigma \\
&= \int_{0}^{a} ((H_{11} + 2 H_{12} + H_{22})(t,a-t))(c + p_{1}t + 
p_{2}(a-t)) \, dt,
\end{split}
\end{equation}
where $d\sigma$ is the contraction of the euclidian volume $dv$ on $\R ^2$ by
$u$.  Note that the integrand (being a rational function of $c$ with a
non-vanishing denominator at $c=0$), and hence the integral, is continuous
near $c=0$; for $c=0$ the integral equals
\[
\tfrac{1}{6} (1 - a) a^3 (-C + 2(p_{1}+ p_{2}) + a (C + 4(p_{1} + p_{2}))),\]
which is clearly negative for $a \in (0,1)$ sufficiently small as long as $C = 4({\bf g} -1)> 2(p_{1} + p_{2})$.
If we take ${\bf g} > 2$, such $p_{1}$ and $p_{2}$  do exist. By Proposition~\ref{p:stability}, this implies  a non-existence result  of extremal 
K\"ahler metrics when $p_{1}$ and $p_{2}$ satisfy the 
above inequality and $c$ is small enough. (As a special case, for
$p_{1}=p_{2}$ we have recast the non-existence part of 
\cite[Thm.~6]{ACGT}.)
\begin{prop}\label{non-existence} Let $M$ be as in
Proposition~\textup{\ref{pr:existence}}, with ${\bf g} >2$ and $p_1,p_2$
satisfying $ 2({\bf g} -1)> p_{1} + p_{2}$. Then all sufficiently `big'
K\"ahler classes do not admit any extremal K\"ahler metric.
\end{prop}

\appendix
\section{Compatible extremal almost K\"ahler metrics}\label{app:almost-kahler}

In this appendix, we calculate the hermitian scalar curvature of a compatible
almost K\"ahler metric and extend the notion of {\it extremal} K\"ahler
metrics to the more general almost K\"ahler case.

Recall that on a general almost K\"ahler manifold $(M^{2m}, g,J, \omega)$, the
{\it canonical hermitian connection} $\nabla$ is defined by
\begin{equation}\label{chern-connection}
\nabla_X Y = D_X Y - \frac{1}{2} J(D_X J)(Y),
\end{equation}
where $D$ is the Levi--Civita connection of $g$. Note that
\begin{equation}\label{nijenhuis}
g((D_X J)Y, Z)  = \frac{1}{2} g(N(X,Y), JZ)
\end{equation}
where $N(X,Y) = [JX,JY] - J[JX,Y] - J[X,JY] -[X,Y]$ is the Nijenhuis tensor of
$J$.  The {\it Ricci form}, $\rho^{\nabla}$, of $\nabla$ represents $2\pi
c_1(M,J)$ and its trace $s^{\nabla}$ (given by $2m \rho^{\nabla}\wedge
\omega^{m-1} = s^{\nabla}\omega^m$) is called {\it hermitian scalar curvature}
of $(g,J,\omega)$.

The hermitian scalar curvature plays an important role in a setting described
by Donaldson~\cite{Do} (see also \cite{gauduchon-book}), in which $s^{\nabla}$
is identified with the momentum map of the action of the group ${\rm
Ham}(M,\omega)$ of hamiltonian symplectomorphisms of a compact symplectic
manifold $(M,\omega)$ on the (formal) K\"ahler Fr\'echet space of
$\omega$-compatible almost K\"ahler metrics ${\mathcal AK}_{\omega}$. It
immediately follows from this formal picture~\cite{draghici,lejmi}  that the critical points of the functional on
${\mathcal AK}_{\omega}$
\[
g \longmapsto \int_M (s^{\nabla})^2 \omega^m
\]
are precisely the $\omega$-compatible almost K\"ahler metrics for which ${\rm
grad}_{\omega} s^{\nabla}$ is a Killing vector field. This provides a natural
extension of the notion of an extremal K\"ahler metric to the more general
almost K\"ahler context.

\begin{defn} An almost K\"ahler metric $(g,\omega)$ for which
${\rm grad}_{\omega} s^{\nabla}$ is a Killing vector field is called {\it
extremal}.
\end{defn}

Now let $M$ be a manifold obtained by the generalized Calabi construction of
Sect.~\ref{s:generalised-calabi}. In the notation of this section, for any
$S^2\mathfrak t^*$-valued function ${\bf H}$ on $\Delta$, satisfying the
boundary and positivity conditions, formulae \eqref{M} introduce a pair
$(g,\omega)$ of a smooth metric $g$ and a symplectic form $\omega$ on $M$,
such that the field of endomorphisms $J$ defined by $\omega(\cdot,
\cdot)=g(J\cdot, \cdot)$ is an almost complex structure, i.e., $(g,\omega)$ is
an almost K\"ahler structure on $M$.\footnote{It is easily seen as in
\cite{Abreu0} that $J$ is integrable, i.e. $(g, \omega)$ defines a K\"ahler
metric, if and only if ${\bf H}^{-1}$ is the hessian of a smooth
function on $\Delta^0$.} We shall refer to such pairs $(g,\omega)$ as {\it
compatible almost K\"ahler metrics} on $M$.

\begin{lemma}\label{lem:hermitian-scalar} The hermitian scalar curvature
$s^{\nabla}$ of a compatible almost K\"ahler metric corresponding to ${\bf
H}=(H_{rs})$ is given by
\[
s^{\nabla}= \sum_{j=1}^Næ \frac{Scal_j}{c_j + \langle p_j, z \rangle} - \frac{1}{p(z)}\sum_{r,s=1}^{\ell} \frac{\partial^2}{\partial z_r \partial z_s} (p(z)H_{rs}).
\]
\end{lemma}
\begin{proof} The result is local and we work on the open dense subset $M^0$
where the $\ell$-torus $\T$ acts freely. Recall that $M^0$ is a principal
$\T^c$ bundle over $\hat S$. Let ${\mathcal V}$ be the foliation defined by
the $\T^c$ fibres and $K_r = J {\rm grad}_g z_r$ be the Killing vector fields
generating $\T$; then $T{\mathcal V}$ is spanned by $K_r, JK_r$ at each point
of $M^0$ and, by construction,
\begin{equation}\label{definition}
{\mathcal L}_{K_r} J =0, \ \ K_r^{\flat}= \sum_s H_{rs}\hat \theta_s, \ \
JK_r^{\flat} = - dz_r.
\end{equation}
In order to compute the hermitian Ricci tensor, we take a local
non-vanishing {\it holomorphic} section $\Phi_{\hat S}$ of the
anti-canonical bundle $K_{\hat S}^{-1}= \wedge^{d,0}(\hat S)$ of $\hat S$
(which pulls back to a $(d,0)$-form on $M$) and wedge it with the
$(\ell,0)$-form $\Phi_{\mathcal V} = (K_1^{\flat} - \sqrt{-1}J K_1^{\flat})
\wedge \cdots \wedge (K_{\ell}^{\flat} - \sqrt{-1}J K_{\ell}^{\flat}).$
Thus, $\Phi = \Phi_{\hat S} \wedge \Phi_{\mathcal V}$ is a non-vanishing
section of $K^{-1}_{M^0}$ and the hermitian Ricci form $\rho^{\nabla}$ is
then given by
\[
\rho^{\nabla} = -d {\mathfrak Im}(\alpha),
\]
where $\nabla \Phi = \alpha \otimes \Phi$.

Denote by $T{\mathcal H}$ the $g$-orthogonal complement of $T\cV$; the spaces
$T{\mathcal H}$ and $T\cV$ then define the decomposition of $TM^0$ as the sum
of horizontal and vertical spaces and, therefore
\begin{equation}\label{basic}
\nabla_X \Phi =( \nabla^{\mathcal H}_X \Phi_{\hat S}) \wedge \Phi_{\mathcal V} + \Phi_{\hat S} \wedge \nabla^{\mathcal V} _X\Phi_{\mathcal V}, 
\end{equation}
where $\nabla_X Y = \nabla^{\mathcal H} _X Y + \nabla^{\mathcal V}_X Y$
denotes the decomposition into horizontal and vertical parts.

Our first observation is that \cite[Prop.~8]{hfkg1} generalizes in the
non-integrable case in the following sense: {\it The foliation ${\mathcal V}$
is totally-geodesic with respect to both the Levi--Civita and hermitian
connections}. Indeed, with respect to the Levi--Civita connection $D$ we have
$\langle D_{K_r} K_s, X \rangle = \langle D_{JK_r} K_s, X \rangle = 0$ for any
$X\in T{\mathcal H}$; using $[K_r, JK_s]=0$, our claim reduces to check that
$\langle D_{JK_r}JK_s, X \rangle =0$. We take $X$ be the horizontal lift of a
basic vector field and use the Koszul formula
\begin{equation*}
\begin{split}
2\langle D_{JK_r}JK_s, X \rangle &= {\mathcal L}_{JK_r} \langle JK_s, X\rangle + {\mathcal L}_{JK_s}\langle JK_r, X\rangle - {\mathcal L}_X \langle JK_r, JK_s \rangle \\
& + \langle [JK_r,JK_s], X\rangle + \langle {\mathcal L}_X JK_r, JK_s \rangle +  \langle {\mathcal L}_X JK_s, JK_r \rangle \\
&= \langle {\mathcal L}_X JK_r, JK_s \rangle +  \langle {\mathcal L}_X JK_s, JK_r \rangle \\
&= ({\mathcal L}_X  g)(JK_r, JK_s) = ({\mathcal L}_X  \hat{g}(z))(JK_r, JK_s)=0,
\end{split}
\end{equation*}
where $(\hat g= {\hat g}(z), \hat \omega= \hat \omega(z))$ denote the K\"ahler
quotient structure on $\hat S$ (also identified with the horizontal part of
$(g,\omega)$).  Considering the hermitian connection $\nabla$, by
\eqref{chern-connection} and \eqref{nijenhuis}, our claim reduces to showing
that $N(K_r, X)$ is horizontal for any $X\in T{\mathcal H}$; using
\eqref{nijenhuis} and the fact that ${\mathcal V}$ is totally-geodesic with
respect to $D$, we get $\langle N(K_r,X), JU\rangle = 2\langle (D_{U} J)
(K_r), X \rangle = 0,$ for any $U \in T{\mathcal V}$.

\smallskip
The observation that $\cV$ is totally geodesic with respect to  $D$ shows that formulae (42)--(46) in \cite[Prop.~9]{hfkg1} hold true in the non-integrable case too,  i.e. we have
\begin{equation}\label{prop.9}
\begin{split}
D_X Y  & = D^{\mathcal H}_X Y - C(X,Y) \\
D_X U &= \langle C(X, \cdot), U \rangle + [X,U]^{\mathcal V} \\
D_U X &= [U, X]^{\mathcal H} + \langle C(X, \cdot), U \rangle \\
D_U V  &= D^{\mathcal V}_U V, 
\end{split}
\end{equation}
where $X, Y \in T{\mathcal H}$, $U, V \in T{\mathcal V}$ and $C(\cdot, \cdot)$ is the O'Neill tensor given by
\[
2C(X,Y)= \sum_{r=1}^{\ell} \Big(\Omega_r (X,Y)K_r + \Omega_r(JX, Y)
JK_r\Big)
\]
with $\Omega_r = d \hat \theta_r = \sum_{j=1}^N p_{jr} \otimes
\omega_j$. Using \eqref{chern-connection}, it follows that the horizontal lift
of $\nabla^{\hat S}$ coincides with the projections of both $D$ and $\nabla$
to horizontal vectors. In particular, for any horizontal lift $X$,
$\nabla^{\mathcal H}_X \Phi_{\hat S} = \frac{1}{2}\Big((d_{\hat S} - \sqrt{-1}
d^c_{\hat S})\log ||\Phi_{\hat S}||^2_{\hat g}\Big)(X) \Phi_{\hat S}.$ On the
other hand, as $K_r$ are Killing and ${\mathcal V}$ is totally geodesic,
$\nabla^{\mathcal V}_X \Phi_{\mathcal V} = 0$, so that we get from
\eqref{basic}
\begin{equation}\label{alpha1}
\alpha(X) = \frac{1}{2}\Big((d_{\hat S} - \sqrt{-1} d^c_{\hat S})\log
||\Phi_{\hat S}||^2_{\hat g}\Big)(X), \ \ \forall X \in T{\mathcal H}.
\end{equation}

\smallskip
To compute $\alpha (U)$ for $U \in T{\mathcal F}$, consider first $\nabla_U
\Phi_{\cV}$. As $\cV$ is totally geodesic, we can write $\nabla_U \Phi_{\cV} =
(a(U) - \sqrt{-1} b(U))\Phi_{\cV}$. It follows from the very definition of
$\Phi_{\cV}$ (and the fact that ${\rm span}(K_1, \cdots, K_{\ell})$ is
$\omega$-Lagrangian) that $\Phi_{\cV}(K_1,K_2, \cdots, K_{\ell}) = \det
g(K_{r},K_{s})= \det{\bf H},$ and therefore
\[
(\nabla_U \Phi_{\cV}) (K_1,K_2, \cdots, K_{\ell})
= \Big(a(U) - \sqrt{-1}b(U)\Big) \det{\bf H}.
\]
Using the definition of $\Phi_{\cV}$ again,  we obtain 
\[
b(U)= {\rm trace} ({\bf H}^{-1} \circ A_U), \ \ (A_U)_{rs} = -\langle\nabla_{U} K_r, JK_s\rangle.
\]
Using that $K_r$ is Killing, \eqref{nijenhuis} and \eqref{definition} we
further calculate
\begin{equation*}
\begin{split}
(A_U)_{rs} & =  -\langle D_U K_r, JK_s \rangle+ \frac{1}{2}\langle (D_U J)(K_r), K_s \rangle \\
            &= \frac{1}{2}\Big(dK_r^{\flat}(JK_s,U) - \frac{1}{2}\omega(N(K_r,K_s), U)\Big) \\
               &= \frac{1}{2} \Big( \sum_{p,k} H_{rk,p} dz_{p}(JK_s)\hat\theta_k(U) - \frac{1}{2} \sum_k dz_k([JK_r,JK_s])\hat \theta_k (U)\Big) \\
               & = \frac{1}{2} \Big( -\sum_{k,p} H_{rk,p} H_{p s}\hat \theta_k(U) - \frac{1}{2} \sum_k dz_k([JK_r,JK_s])\hat \theta_k(U) \Big) \\
               &= - \frac{1}{4} \sum_{k,p} (H_{rk,p} H_{p s} + H_{sk, p}H_{pr}) \hat \theta_k(U),
\end{split}
\end{equation*}
so that
\begin{equation}\label{alpha2}
b(U)= \sum_{r,s} H^{rs}(A_U)_{rs}= -\frac{1}{2}\sum_{r,k} H_{rk,r}\hat \theta_k(U).
\end{equation}

\smallskip
Finally, in order to compute $\nabla_U \Phi_{\hat S}$, note that ${\mathcal
L}_U \Phi_{\hat S} =0$, and therefore
\begin{equation*}
\begin{split}
(\nabla_U \Phi_{\hat S})(X_1, \cdots, X_d) &= \sum_{k=1}^d \Big(\frac{1}{2}\Phi_{\hat S}(X_1, \cdots, X_{k-1}, J(D_U J)(X_k), X_{k+1}, \cdots, X_d)  \\
\ & \ \ \ \  \ \ \ -\Phi_{\hat S}(X_1, \cdots, X_{k-1}, (D^{\mathcal H}_{X_k}U), X_{k+1}, \cdots, X_d) \Big),
\end{split}
\end{equation*} 
where $X_k \in T{\mathcal H}$.  Now, using \eqref{prop.9}, we further specify
\begin{equation*}
\begin{split}
& (D^{\mathcal H}_{X_k}U) = \frac{1}{2}\sum_{r=1}^{\ell} \sum_{j=1}^N\Big(K_r^{\flat}(U) p_{jr}(J_jX^j_k) - JK_r^{\flat}(U)p_{jr}X^j_k \Big),\\
& \Big((D_U J)(X_k)\Big)^{\mathcal H} =0,
\end{split}
\end{equation*}
where $X_k^j$ (resp. $J_j$) denote the $g_{\hat S}$-orthogonal projection
(resp. restriction) of $X_k$ (resp. $J$) to the subspace $TS_j \subset T\hat
S$ (recall that the universal cover of $(\hat S, g_{\hat S})$ is the K\"ahler
product of $(S_j, g_j, \omega_j)$, so that the projection of $TS_j$ to $T\hat
S$ is a well-defined $D$-parallel subbundle of $T{\hat S}$). Using
\eqref{definition}, and the expressions \eqref{alpha1} and \eqref{alpha2}, we
eventually find that
\begin{equation*}
\begin{split}
{\mathfrak Im}(\alpha) & = - \frac{1}{2} d_{\hat S}^c \log ||\Phi_{\hat S}||^2_{\hat g}  +\frac{1}{2}d^c \log p(z) + \frac{1}{2} \sum_{k,r} H_{kr,k} {\hat \theta}_r \\
&= - \frac{1}{2} d_{\hat S}^c \log ||\Phi_{\hat S}||^2_{\hat g} + \frac{1}{2p(z)}\sum_{k,r} \Big((\frac{\partial p} {\partial z_k}) H_{kr} + p(z)\frac{\partial H_{kr}}{\partial z_k}\Big)\hat \theta_r \\
\rho^{\nabla} & = \sum_{j=1}^N \rho_j  - \sum_{i,r,k} \frac{\partial}{\partial z_k} \Big( \frac{1}{2p(z)}  \frac{\partial (p(z)H_{ir})}{\partial z_i}\Big) dz_k \wedge \theta_r \\ & \ \ -  \frac{1}{2p(z)} \sum_{i,r} \frac{\partial (p(z)H_{ir})}{\partial z_i}\frac{\partial {\hat \omega}}{\partial z_r},
\end{split}
\end{equation*}
where, we recall, $\rho_j$ is the Ricci form of $(S_j,g_j,\omega_j)$, $\hat
\omega (z) = \sum_{j=1}^N\Big(\sum_{r=1}^{\ell} (p_{jr}z_r + c_j)
\omega_j\Big)$, and $p(z) =\prod_{j=1}^N \Big(\sum_{r=1}^{\ell} p_{jr}z_r +
c_j\Big)^{d_j}$.  The formula for $s^{\nabla}$ follows easily.
\end{proof}

\end{document}